\documentclass[12pt]{amsart}
\usepackage{amssymb, a4wide, mathdots,url,hyperref, shuffle}
\usepackage{bbm}
\usepackage{pb-diagram}

\renewcommand{\Re}{\operatorname{Re}}

\newcommand{\Hom}{\operatorname{Hom}}
\newcommand{\Ext}{\operatorname{Ext}}
\renewcommand{\subset}{\subseteq}

\newtheorem{theorem}{Theorem}[section]
\newtheorem{lemma}[theorem]{Lemma}
\newtheorem{proposition}[theorem]{Proposition}
\newtheorem{remark}[theorem]{Remark}
\newtheorem{conjecture}[theorem]{Conjecture}
\newtheorem{definition}[theorem]{Definition}
\newtheorem{corollary}[theorem]{Corollary}

\title[non-generic local Gan-Gross-Prasad]{On restriction of unitarizable representations of general linear groups and the non-generic local Gan-Gross-Prasad conjecture}
\author{Maxim Gurevich}
\date{\today}

\newcommand{\gotM}{\mathfrak{m}}
\newcommand{\gotN}{\mathfrak{n}}

\DeclareMathOperator{\irr}{Irr}
\DeclareMathOperator{\seg}{Seg}
\DeclareMathOperator{\supp}{supp}
\DeclareMathOperator{\multi}{Mult}
\DeclareMathOperator{\im}{Im}
\DeclareMathOperator{\sspan}{span}

\begin{document}

\address{Department of Mathematics, Technion - Israel Institute of Technology, Haifa, Israel, 3200003.}
\email{maxg@technion.ac.il}

\begin{abstract}
We prove the first direction of a recently posed conjecture by Gan-Gross-Prasad, which predicts branching laws that govern restriction from $p$-adic $GL_n$ to $GL_{n-1}$ of irreducible smooth representations within the Arthur-type class.

We extend this prediction to the full class of unitarizable representations, by exhibiting a combinatorial relation that must be satisfied for any pair of irreducible representations, in which one appears as a quotient of the restriction of the other.

We settle the full conjecture for the cases in which either one of the representations in the pair is generic.

The method of proof involves a transfer of the problem, using the Bernstein decomposition and the quantum affine Schur-Weyl duality, into the realm of quantum affine algebras. This restatement of the problem allows for an application of the combined power of a result of Hernandez on cyclic modules together with the Lapid-M\'{i}nguez criterion from the $p$-adic setting.

\end{abstract}

\maketitle
\section{Introduction}

Let $\pi$ be a smooth irreducible representation of the group $GL_n(F)$, where $F$ is a $p$-adic field. Let us consider $GL_{n-1}(F)$ as a subgroup of $GL_n(F)$, embedded in a natural way, which in matrix form is described as
\[
\left( \begin{array}{cc} GL_{n-1}(F) & 0_{n-1,1} \\ 0_{1,n-1} & 1 \end{array} \right) < GL_n(F)\;.
\]
We study the branching laws that govern the decomposition of the restricted representation $\pi|_{GL_{n-1}(F)}$ into irreducible representations of $GL_{n-1}(F)$. As an approachable goal, we focus on the description of the possible isomorphism classes of irreducible representations which appear as \textit{quotients} of $\pi|_{GL_{n-1}(F)}$.

One fundamental result in that direction was achieved in \cite{aiz-gur2}, where it was shown that
\[
\Hom_{GL_{n-1}(F)}(\pi,\sigma)
\]
is at most one-dimensional, for all $\pi \in \irr GL_n(F)$, $\sigma\in \irr GL_{n-1}(F)$.

As a step forward, we are seeking for a meaningful description of the collection of pairs $(\pi,\sigma)$, for which the latter morphism space is non-zero.

An irreducible representation is said to be generic, if it can be produced on a Whittaker model of functions. It is a classical fact proved in \cite{jpss} through the study of $L$-functions, that for every pair of generic representations $(\pi,\sigma)$ as above, $\Hom_{GL_{n-1}(F)}(\pi,\sigma)\neq 0$.

In later years, great efforts were focused on posing and proving analogous rules for (quasi-split) classical groups, in place of the general linear group \cite{gp1,gp2,ggp1,ggp2}. The resulting branching laws became known as the \textit{local Gan-Gross-Prasad} conjectures. In similarity with $GL_n$, these laws for classical groups were always set to pertain the generic case, in the sense of generic Langlands parameters which parameterize irreducible representations.

Back in the $GL_n$ case, an effective answer for the general restriction problem is still largely considered unpractical. Yet, recently Gan-Gross-Prasad \cite{ggp-non} have revisited this problem in an attempt to formulate branching laws that would extend beyond the generic case.

They stipulated a principle that clear combinatorial rules should describe the pairs $(\pi,\sigma)$ with non-zero $\Hom_{GL_{n-1}(F)}(\pi,\sigma)$, when $\pi$ and $\sigma$ belong to a well-behaved class of representations. More precisely, they formulated a conjecture which concerns a subclass of unitarizable representations which is described by \textit{Arthur parameters}.

An Arthur parameter for $GL_n(F)$, in the definition of \cite{ggp-non}, stands for an admissible homomorphism
\[
\phi: W_F \times SL_2(\mathbb{C}) \times SL_2(\mathbb{C}) \to GL_n(\mathbb{C})\;,
\]
such that the image of $W_F$ is bounded, and the restriction of $\phi$ to each $SL_2(\mathbb{C})$ component is algebraic. Here $W_F$ stands for the Weil group of the field $F$.

In particular, an Arthur parameter $\phi$ is a completely reducible representation, which decomposes as
\[
\phi = \oplus_{i=1}^k \psi_i \otimes V_{a_i} \otimes V_{b_i}\;,
\]
where $\{\psi_i\}$ are irreducible representations of $W_F$ with bounded image, and $V_d$, for $d\in\mathbb{Z}_{>0}$, denotes the unique isomorphism class of a $d$-dimensional irreducible algebraic representation of $SL_2(\mathbb{C})$.

Given an Arthur parameter $\phi$, one can attach a \textit{$L$-parameter} to it. Consequently, by the established local Langlands reciprocity this $L$-parameter gives rise to an irreducible representation $\pi(\phi)$ of $GL_n(F)$. We will say that a representation which is constructed in this manner is of \textit{Arthur-type}.

\begin{conjecture}[\cite{ggp-non}]\label{main-conj}
Suppose that
\[
\phi_1=  \oplus_{i=1}^k \psi_i \otimes V_{a_i} \otimes V_{b_i}\;,\quad\phi_2=  \oplus_{i=1}^l \psi'_i \otimes V_{a'_i} \otimes V_{b'_i}\;,
\]
are two Arthur parameters for $GL_n(F), GL_{n-1}(F)$, respectively.

Then,
\[
\Hom_{GL_{n-1}(F)} \left(\pi(\phi_1)|_{GL_{n-1}(F)},\,\pi(\phi_2)\right) \neq 0
\]
holds, if and only if, there are disjoint partitions
\[
\{1,\ldots,k\} = I_1\cup I_2\cup I_3, \quad\{1,\ldots,l\} = J_1\cup J_2\cup J_3\;,
\]
and bijections $u:I_1\to J_2$, $d:I_2\to J_1$, which satisfy
\[
(a'_{u(i)}, b'_{u(i)}) = (a_i,b_i+1),\,\psi'_{u(i)}\cong \psi_i\; \forall i\in I_1,
\]
\[
(a'_{d(i)},b'_{d(i)}) = (a_i,b_i-1),\,\psi'_{d(i)}\cong \psi_i\; \forall i\in I_2,\;
\]
\[
b_i =1 \; \forall i\in I_3,\quad b'_j=1\;\forall j\in J_3\;.
\]
\end{conjecture}


The results of this article are as follows:
\begin{itemize}
  \item Theorem \ref{thm-main}, which provides a full proof of one direction of Conjecture \ref{main-conj} through the use of tools from the representation theory of quantum affine algebras.
  \item Theorem \ref{thm-main-ext}, which extends the conjecture's statement to the class of all unitarizable representations.
  \item Theorem \ref{thm-conv}, which settles some families of cases of the converse direction of the conjecture.
\end{itemize}
Combining these results, we settle both directions for the case when at least one of $\pi(\phi_1), \pi(\phi_2)$ is generic.

One motivation for the formulation of Conjecture \ref{main-conj} were the works of Clozel \cite{cloz}, Venkatesh \cite{venk-bs} and Lapid-Rogawski \cite{lapid-rog} in the setting of \textit{unitary} representations. They studied a restriction problem in the sense of direct integral decomposition, and showed that the Burger-Sarnak principle for automorphic representations implies some necessary combinatorial conditions to occur in a restriction.

Even though there is no immediate relation between the smooth and unitary problems, our Theorem \ref{thm-main} is consistent with \cite[Proposition 2(1)]{venk-bs} in a certain sense. Namely, the cases where $I_1 = J_2=\emptyset$ hold in the statement of Conjecture \ref{main-conj} produce a refinement of the $SL(2)$-type condition required by the mentioned proposition from \cite{venk-bs}. Further discussion of this relation can be found in \cite{ggp-non}.

\subsection{Outline of proof}

When attempting to tackle Conjecture \ref{main-conj} with standard techniques in hand, it appears unavoidable to encounter the intricacies of the structure of the Bernstein-Zelevinski product. As explain below, we are able to isolate the required information on such products in the form of Theorem \ref{thm:intro2}.

We next prove Theorem \ref{thm:intro2} by transferring it into other Lie-theoretic Abelian categories. Namely, we move into the representation theories of affine Hecke algebras and quantum affine algebras of type $A$.

Let us discuss the steps of the proof in more detail.

We first make use of the classical Bernstein-Zelevinski filtration for the space of $\pi|_{GL_{n-1}}$. This special feature of general linear groups allows for a translation (Proposition \ref{filt-frob}) of the restriction problem into questions on spaces of the form
\begin{equation}\label{intro-eq}
\Hom_{GL_{n-i}(F)} (\nu^{1/2}\otimes\pi^{(i)},\,^{(i-1)}\sigma),\;
\end{equation}
where $\nu^{1/2}$ is a certain character, and $\pi \mapsto \pi^{(i)}$ and $\pi\mapsto \,^{(i)}\pi$ are the Bernstein-Zelevinski derivative functors, which attach finite-length representations of a smaller rank group to a given irreducible representation.

Thus, we are left with questions on morphism spaces inside a category of finite-length representations. Moreover, derivatives of Arthur-type representations are built out of (Bernstein-Zelevinski) products of derivatives of \textit{Speh} representations. We call these derived Speh representations (which happen to be irreducible) \textit{quasi-Speh} representations.

The reasoning portrayed thus far was employed already in \cite{ggp-non} to tackle the problem and to prove some basic cases of Conjecture \ref{main-conj} (such as when $\pi(\phi_1),\pi(\phi_2)$ are Speh representations themselves).


Yet, the spaces \eqref{intro-eq} for the general case of Arthur-type representations were discovered to be substantially more intricate: Products of quasi-Speh representations are often reducible. The sub-representation structure of such Bernstein-Zelevinski products is often a highly non-trivial issue, as evidenced by a batch of recent works (\cite{LM2, LM3, tadic-speh, me-quantum}).

Such difficulties are manifested in our problem through the following key statement, on which the resolution of the first direction of Conjecture \ref{main-conj} (Theorem \ref{thm-main}) is highly dependent.

\begin{theorem}\label{thm:intro2}
  (restatement of Proposition \ref{main-prop})
For any choice of quasi-Speh representations $\pi_1,\ldots,\pi_k$, there is an ordering $\omega$ (permutation of $\{1,\ldots,k\}$), for which the Bernstein-Zelevinski product representation
\[
\pi_{\omega(1)}\times\cdots\times \pi_{\omega(k)}
\]
has a unique irreducible quotient, whose Langlands parameter is given as the sum of Langlands parameters of $\pi_1,\ldots,\pi_k$.
\end{theorem}

We show (Proposition \ref{replace}) that the case of $k=2$ in Theorem \ref{thm:intro2} follows from the recent work of Lapid-M\'{i}nguez \cite{LM2} on behavior of products of representations in the \textit{ladder} class.

As a consequence, we are now left with the problem of whether Theorem \ref{thm:intro2} can be settled by looking on products of pairs of representations. Namely, we prove the following general phenomenon.

\begin{theorem}\label{thm:intro3}
Suppose that $\sigma_1,\ldots,\sigma_k$ are irreducible smooth representations of general linear groups, such that for all $1\leq i < j \leq k$, the product representation $\sigma_i\times \sigma_j$ has a unique irreducible quotient whose Langlands parameter is given as the sum of Langlands parameters of $\sigma_i$ and $\sigma_j$.

Then, the product representation
\[
\sigma_1 \times \cdots \times \sigma_k
\]
has a unique irreducible quotient, whose Langlands parameter is given as the sum of Langlands parameters of $\sigma_1,\ldots,\sigma_k$.
\end{theorem}

The mechanism of \textit{Bernstein decomposition} (Section \ref{sect-equiv-heck}) presents the category of smooth representations of $GL_n(F)$ as a product of smaller Abelian categories called Bernstein blocks. It is sufficient to prove Theorem \ref{thm:intro3} for each such block.

It is well known that each block can be described as modules over a complex algebra. The identification of these algebras with \textit{affine Hecke algebras} (for $GL_n$) was done in \cite{bk-ss, bk-book} through type theory, and independently in \cite{heier-cat} through a more explicit approach. Thus, it is enough to solve the problem of Theorem \ref{thm:intro3} for modules over affine Hecke algebras (stated as Theorem \ref{transl}).

Next, we are able to pass from modules over affine Hecke algebras to modules over quantum affine algebras by using the \textit{quantum affine Schur-Weyl duality} functors (Section \ref{sect-sw-dual}), developed by Chari-Pressley \cite{cp-duality}. Since all functors involved turn out to be monoidal, in a suitable sense, we are finally left with a statement regarding the simple quotients of tensor products of quantum affine modules.


The advantage of posing our problem in a language of quantum groups lies in a recent result of Hernandez \cite{hernan-cyc} dealing with \textit{cyclic} modules. These are not necessarily simple modules, that are generated by their highest weight vector. It states that a product $V_1\otimes\ldots \otimes V_k$ of simple modules $V_1,\ldots, V_k$ is cyclic, if all the products $V_i\otimes V_j\;,1\leq i < j\leq k,$ are cyclic.

When transferring the notion of a cyclic products back to the $p$-adic setting through our sequence of functors, we see that the theorem of Hernandez gives the precise statement of Theorem \ref{thm:intro3}.

We comment that in a work under preparation of Alberto M\'{i}nguez together with the author, we were able to extract to the key arguments of \cite{hernan-cyc} and use them to construct an alternative ``purely $p$-adic" proof of Theorem \ref{thm:intro3}.



\subsection{Paper structure}
Section \ref{sect-backg} surveys the basic tools needed to study the smooth representation theory of $p$-adic $GL_n$, together with some necessary lemmas. In particular, we recall the Zelevinski multisegment parametrization, which is an essential tool in our work. We make note of the basic Proposition \ref{disjsupp}, whose statement has not appeared previously in the literature to the best of our knowledge. Section \ref{sect-lm} recovers the necessary results from \cite{LM2}.

Section \ref{sect-hecke} portrays the categorical passage from representations of $GL_n(F)$ into those of affine Hecke algebras. Proposition \ref{equiv-block} delves into the resulting correspondence between irreducible representations. The proof of the key Theorem \ref{transl} is delayed to Section \ref{sect-qa}.

Section \ref{sect-work} contains the gist of the combinatorial work in the class of quasi-Speh representations and their products, which is necessary for the proof of Conjecture \ref{main-conj}.

In Section \ref{sect-main} we prove the main theorems discussed above.

Finally, Section \ref{sect-qa} surveys the necessary ingredients from the representations theory of $U_q\left(\hat{\frak{sl}}_N\right)$ with the aim of proving Theorem \ref{transl}, which is shown to be essentially a translation of the main result of \cite{hernan-cyc}.

\subsection{Acknowledgements}
I would like to thank Wee Teck Gan, for sharing this problem with me and whose insights and optimism are an invaluable guide. Thanks are due to Erez Lapid for the continuing encouragement and useful conversations, to David Hernandez for sharing his results and expertise, to Dipendra Prasad for sharing his views on the problem and to Kei Yuen Chan and Gordan Savin for sharing their preprint and their views on similar themes.

This work was begun while the author was a research fellow at the National University of Singapore, supported by Wee Teck Gan’s MOE Tier 2 grant 146-000-233-112.

\section{Background on representation theory of $p$-adic $GL_n$}\label{sect-backg}

Let $F$ be a $p$-adic field. Unless explicitly stated, $F$ will be fixed and omitted from our notation. We are interested in the representation theory of the groups $G_n := GL_n(F)$, for all $n\geq1$.

For a given $n$, let $\alpha = (n_1, \ldots, n_r)$ be a composition of $n$. We denote by $M_\alpha$ the subgroup of $G_n$ isomorphic to $G_{n_1} \times \cdots \times G_{n_r}$ consisting of matrices which are diagonal by blocks of size $n_1, \ldots, n_r$ and by $P_\alpha$ the subgroup of $G_n$ generated by $M_\alpha$ and the upper
unitriangular matrices. A standard parabolic subgroup of $G_n$ is a subgroup of the form $P_\alpha$ and its standard Levi factor is $M_\alpha$.

For a $p$-adic group $G$, let $\mathfrak{S}(G)$ be the category
of smooth complex representations of $G$, and $\mathfrak{R}(G)$ be the sub-category of objects of finite length. Denote by $\irr(G)$ the set of equivalence classes
of irreducible objects in $\mathfrak{R}(G)$. Denote by $\mathcal{C}(G)\subset \irr(G)$ the subset of irreducible supercuspidal representations.

We write $\mathbf{i}_\alpha: \mathfrak{R}(M_\alpha)\to \mathfrak{R}(G_n)$ for the parabolic induction functor associated to $P_\alpha$.

For $\pi_i\in \mathfrak{R}(G_{n_i})$, $i=1,\ldots,r$, we write
\[
\pi_1\times\cdots\times \pi_r := \mathbf{i}_{(n_1,\ldots,n_r)}(\pi_1\otimes\cdots\otimes \pi_r)\in \mathfrak{R}(G_{n_1+\ldots+n_r}).
\]

We also write $\irr = \cup_{m\geq0} \irr(G_m)$ and $\mathcal{C} = \cup_{m\geq1} \mathcal{C}(G_m)$.

For any $n$, let $\nu^s= |\det|^s_F,\;s\in \mathbb{C}$ denote the family of one-dimensional representations of $G_n$, where $|\cdot|_F$ is the absolute value of $F$. For $\pi\in \mathfrak{R}(G_n)$, we write $\nu^s\pi=\pi\nu^s := \pi\otimes \nu^s\in \mathfrak{R}(G_n)$.

The map $s\mapsto \nu^s$ is a group homomorphism from $\mathbb{C}$ to the group of characters of $G_n$, whose kernel is $\frac{2\pi i}{\log q_F}\mathbb{Z}$, where $q_F$ is the residue characteristic of $F$.

The group of (unramified) characters $\{\nu^s\,:\, s\in\mathbb{C}\}$ of $G_n$ acts on $\mathcal{C}(G_n)$ by $\rho\mapsto \rho\nu^s$. For $\rho\in\mathcal{C}(G_n)$, we write $o(\rho)$ for the (finite) order of the stabilizer of $\rho$ for that action.

For $\pi\in \irr$, the central character $\chi_\pi$ of $\pi$ is a representation of $G_1$. Hence, $|\chi_\pi| = \nu^{s_\pi}$, for a number $s_\pi\in \mathbb{R}$. We will write $\Re(\pi) = s_\pi\in \mathbb{R}$.

Given a set $X$, we write $\mathbb{N}(X)$ for the commutative monoid of maps from $X$ to $\mathbb{N}= \mathbb{Z}_{\geq0}$ with finite support. We will sometimes treat an element $A\in \mathbb{N}(X)$ as a finite subset of $X$, by writing $x\in A$ for a given $x\in X$, in case $A(x)>0$.

There is a natural embedding $X\to \mathbb{N}(X)$ which sends an element to its indicator function. Thus, we will often treat elements of $X$ as elements in $\mathbb{N}(X)$.

\subsection{Multisegments}\label{sect-multis}

The elements of $\irr$ are classified by \textit{multisegments} in the following manner.

A \textit{segment} $\Delta= [a,b]_\rho$ is a formal object defined by a triple $([\rho],a,b)$, where $\rho\in \mathcal{C}$ and $a\leq b$ are two integers, up to the equivalence $[a,b]_\rho=[a',b']_{\rho'}$, when $\rho\nu^a\cong \rho'\nu^{a'}$ and $\rho\nu^b \cong \rho'\nu^{b'}$.

It is also useful to refer to the empty segment given in the form $[a,a-1]_\rho$, for any $\rho\in \mathcal{C}$ and integer $a$.

For a segment $\Delta=[a,b]_\rho$, we will write $b_\rho(\Delta)=a$ and $e_\rho(\Delta)=b$.

A segment $\Delta_1$ is said to precede a segment $\Delta_2$, if $\Delta_1 = [a_1,b_1]_{\rho} ,\;\Delta_2= [a_2,b_2]_{\rho}$ and $a_1\leq a_2-1\leq b_1<b_2$. We will write $\Delta_1 \prec \Delta_2$ in this case.

Let $\seg$ denote the collection of all segments. Elements of $\multi: = \mathbb{N}(\seg)$ are called multisegments.

The \textit{Zelevinski classification} \cite{Zel} defines a bijection
\[
Z: \multi \to \irr\;.
\]

Let us briefly recall the process that constructs $Z$. First, $Z(\Delta)$ is defined directly for every segment $\Delta$. It is easily verified that each $\gotM\in \multi$ can be written additively as $\gotM = \Delta_1 +\ldots + \Delta_k$, where $\Delta_i$ are segments whose order is chosen so that $\Delta_i\nprec \Delta_j$, for any $i<j$. Next, the \textit{standard} representation
\[
\zeta(\gotM)= Z(\Delta_1)\times\cdots \times Z(\Delta_k)
\]
is constructed. While $\zeta(\gotM)$ may be reducible, its unique irreducible subrepresentation is defined to be $Z(\gotM)$.

We say that $\gotN \leq \gotM$, if the isomorphism class of $Z(\gotN)$ appears as a subquotient in $\zeta(\gotM)$. It is shown in \cite{Zel} that this is a partial order on $\multi$ (the Zelevinski order).

Similarly, the \textit{Langlands (quotient) classification}, which describes irreducible representations of any reductive $p$-adic group, can be stated for the case of the groups $\{G_n\}_{n=1}^\infty$, as another\footnote{
In fact, the Zelevinski classification can also be obtained as a variation of the Langlands classification, when considering the anti-tempered spectrum of the group, rather than the tempered spectrum.} bijection
\[
L: \multi \to \irr\;.
\]

The relation between both bijections $Z$ and $L$ is further discussed in Section \ref{sect-conseq}.

For a segment $\Delta = [a,b]_\rho$ and $s\in\mathbb{C}$, we write $\Delta\nu^s :=[a,b]_{\rho\nu^s}$. We naturally extend this to an operation $\gotM \mapsto \gotM\nu^s$ on $\multi$. It is easy to check that, $Z(\gotM\nu^s) = Z(\gotM)\nu^s$ and $L(\gotM\nu^s) = L(\gotM)\nu^s$ hold.

Let $\pi_1= Z(\gotM_1),\ldots,  \pi_t= Z(\gotM_t) \in \irr$ be given. Then $\pi_1\times\cdots\times \pi_t$ contains $Z(\gotM_1+\ldots+\gotM_t)$ with multiplicity one in its composition series (e.g. \cite[Proposition 2.5(5)]{LM2}). In particular, when $\pi_1\times\cdots\times \pi_t$ is irreducible, we must have $\pi_1\times\cdots\times\pi_t\cong Z(\gotM_1+\ldots+\gotM_t)$.

The analogous statements remain true when replacing the $Z$ bijection with $L$.

\begin{lemma}\label{mult-techlem}
Let $\pi_1 = L(\gotM_1),\ldots, \pi_k = L(\gotM_k)$ and $\sigma_1 = L(\gotN_1),\ldots,\sigma_l = L(\gotN_l)$ be irreducible representations, for which $L(\gotM_1+\ldots + \gotM_k)$ is the unique irreducible quotient of $\pi_1\times\cdots\times \pi_k$ and $L(\gotN_1+\ldots + \gotN_l)$ is the unique irreducible subrepresentation of $\sigma_1\times\cdots\times \sigma_l$.

Suppose that
\[
\Hom (\pi_1\times\cdots\times \pi_k, \sigma_1\times\cdots\times \sigma_l)\neq \{0\}\;.
\]
Then, $\gotM_1+\ldots + \gotM_k = \gotN_1+\ldots + \gotN_l$.

\end{lemma}
\begin{proof}
Let $0\neq f$ be a homomorphism in $\Hom (\pi_1\times\cdots\times \pi_k, \sigma_1\times\cdots\times \sigma_l)$, and let $\kappa$ denote its image. By assumption $\kappa$ contains $L(\gotM_1+\ldots + \gotM_k)$ as a quotient, and $L(\gotN_1+\ldots + \gotN_l)$ as a subrepresentation. In particular, we see that $L(\gotM_1+\ldots + \gotM_k)$ is a subquotient of $\sigma_1\times\cdots\times \sigma_l$.

It is known that for every subquotient $L(\mathfrak{t})$ of $\sigma_1\times\cdots\times \sigma_l$, we have $\mathfrak{t}\leq \gotN_1+\ldots +\gotN_l$.

Hence, $\gotM_1+\ldots + \gotM_k \leq \gotN_1+\ldots + \gotN_l$. We claim analogously that $\gotN_1+\ldots + \gotN_l \leq \gotM_1+\ldots + \gotM_k$ and the statement follows.

\end{proof}

Given a supercuspidal representation $\rho\in \mathcal{C}$, let $\seg_\rho$ denote the collection of segments of the form $[a,b]_{\rho\nu^s}$, for some integers $a,b$ and $s\in \mathbb{C}$.

We also define the sub-monoid $\multi_\rho: = \mathbb{N}(\seg_\rho)\subset \multi$. When $\nu^0\in \irr(G_1)$ is the trivial representation (thought of as a supercuspidal representation), we write $\multi_0 = \multi_{\nu^0}$.

We will also let the field $F$ vary for this part of the discussion, and write $\seg^F_\rho$, $\multi^F_\rho$, $\multi^F_0$ for the corresponding objects, defined for an arbitrary $p$-adic field $F$.

Given $\rho_1\in \mathcal{C}(GL_{n_1}(F_1))$ and $\rho_2\in \mathcal{C}(GL_{n_2}(F_2))$ with $q_{F_1}^{o(\rho_1)}= q_{F_2}^{o(\rho_2)}$, we have a well-defined bijection
\[
\phi_{\rho_1,\rho_2}: \seg^{F_1}_{\rho_1} \to \seg^{F_2}_{\rho_2},\quad \phi([a,b]_{\rho_1\nu^s}) =  [a,b]_{\rho_2\nu^s}\;,
\]
which naturally extends to an isomorphism of monoids $\phi_{\rho_1,\rho_2}: \multi^{F_1}_{\rho_1} \to \multi^{F_2}_{\rho_2}$.

\subsection{Gelfand-Kazhdan involution}\label{sect-gk}

The outer automorphism $g\mapsto (g^t)^{-1}$ on $G_n$, gives an involutive auto-equivalence $\eta$ of the category $\mathfrak{R}(G_n)$. Let us write $\overline{\eta}$ for the composition of $\eta$ with the operation of taking the contragredient (smooth dual) representation.

The resulting involution $\overline{\eta}$ is a contragredient functor with satisfies the special property $\overline{\eta}(\pi)\cong \pi$, for all $\pi\in \irr$, as shown by a classical result of Gelfand-Kazhdan \cite{GK75}.

It is easy to check that $\overline{\eta}$ is compatible with the parabolic induction product, in the sense of
\[
\overline{\eta}(\pi_1\times\cdots\times \pi_t)\cong \overline{\eta}(\pi_t)\times \cdots\times \overline{\eta}(\pi_1)\;,
\]
for any $\pi_i\in \mathfrak{R}(G_{n_i})$, $i=1,\ldots,t$.

The existence of $\overline{\eta}$ also gives the following well-known corollary.
\begin{proposition}\label{GK}
For all $\pi_1,\ldots,\pi_t \in \irr$, the socle (maximal semisimple subrepresentation) of $\pi_1\times \cdots\times \pi_t$ is isomorphic to the co-socle (maximal semisimple quotient) of $\pi_t\times \cdots\times \pi_1$ .

In particular, when $\pi_1\times \cdots\times \pi_t$ is irreducible, it is isomorphic to $\pi_{w(1)}\times \cdots\times \pi_{w(t)}$, for any permutation $w$ on $\{1,\ldots,t\}$.

\end{proposition}

\subsection{Supercuspidal support}

For every $\pi\in \irr$ there exist $\rho_1,\ldots,\rho_r \in\mathcal{C}$ for which $\pi$ is a sub-representation of $\rho_1\times\cdots\times \rho_r$. The notion of \textit{supercuspidal support} can then be defined as the multiset
\[
\supp(\pi) \in \mathbb{N}(\mathcal{C})\;,
\]
given by the tuple $(\rho_1,\ldots,\rho_r)$.

It is known that for any $\pi_1,\pi_2\in \irr$, the product $\pi_1\times \pi_2$ is irreducible, unless there is a supercuspidal representation $\rho\in \supp(\pi_1)$, so that either $\rho\nu^1\in \supp(\pi_2)$ or $\rho\nu^{-1}\in \supp(\pi_2)$.

For $\sigma\in \mathfrak{R}(G_n)$, we will write $\supp(\sigma) = I$, in case $\supp (\pi)= I$ holds, for all irreducible sub-quotients $\pi$ of $\sigma$.

Recall from the general theory of the Bernstein center, that any representation $\sigma\in \mathfrak{R}(G_n)$ splits uniquely to the form $\sigma = \oplus_I \sigma_I$, where the sum goes over distinct multisets $I\in \mathbb{N}(\mathcal{C})$, so that $\supp(\sigma_I) = I$.

In particular, for $\pi,\sigma\in \mathfrak{R}(G_n)$, we have $\Hom (\pi,\sigma) = \oplus_I \Hom (\pi_I,\sigma_I)$.

\begin{proposition}\label{disjsupp}
Suppose that $\pi_i, \sigma_i \in \mathfrak{R}(G_{n_i})$ are given for $i=1,\ldots,t$, such that $\supp(\pi_i) = \supp(\sigma_i) = I_i$, for pairwise disjoint multisets $I_1,\ldots,I_t\in \mathbb{N}(\mathcal{C})$.

Then, we have a natural identification

\[
\Hom( \pi_1\times\cdots\times \pi_t, \sigma_1\times\cdots\times \sigma_t) \cong \Hom(\pi_1,\sigma_1) \otimes\cdots\otimes \Hom (\pi_t ,\sigma_t) \;.
\]

\end{proposition}
\begin{proof}
Consider the parabolic induction functor $\mathbf{i}_\alpha: \mathfrak{R}(G_{n_1}\times \cdots\times G_{n_t})\to \mathfrak{R}(G_N)$ ($N=n_1+\ldots+n_t$). It clearly gives an embedding of the space $\otimes_{i=1}^t \Hom(\pi_i,\sigma_i)$ into $\Hom( \pi_1\times\cdots\times \pi_t, \sigma_1\times\cdots\times \sigma_t)$. We are left to show that any morphism in the latter space is contained in the image of $\mathbf{i}_\alpha$.

Recall that parabolic induction has an exact left-adjoint functor $\mathbf{r}_\alpha$, that is, the Jacquet functor. This adjunction gives a canonical morphism $\epsilon_{\underline{\pi}}: \mathbf{r}_\alpha(\pi_1\times\cdots\times \pi_t) \to \pi_1\otimes \cdots\otimes \pi_t$. Similarly, we have $\epsilon_{\underline{\sigma}}: \mathbf{r}_\alpha(\sigma_1\times\cdots\times \sigma_t) \to \sigma_1\otimes \cdots\otimes \sigma_t$

The suitable Mackey theory in the form of the geometric lemma of Bernstein-Zelevinski (\cite{BZ1}, or see \cite[1.2]{LM2}) predicts a more precise information on $\epsilon_{\underline{\pi}}$.

First, $\epsilon_{\underline{\pi}}$ is surjective. Second, for any irreducible subquotient $\tau = \tau_1\otimes\cdots\otimes \tau_t$ of $\mathbf{r}_\alpha(\pi_1\times\cdots\times \pi_t)$, there is a matrix of multisets $(J_i^j)_{i,j=1}^t$ in $\mathbb{N}(\mathcal{C})$, such that $I_i = J_i^1 + \ldots + J_i^t$, and $\supp(\tau_j) = J_1^j + \ldots + J_t^j$.

Moreover, when $\tau$ is a sub-quotient of $\ker \epsilon_{\underline{\pi}}$, we know that $J_i^j\neq 0$, some $i\neq j$. In particular, by the disjointness assumption $(\supp(\tau_1),\ldots, \supp(\tau_t)) \neq  (I_1,\ldots, I_t)$ in that case. Hence, the projection $\epsilon_{\underline{\pi}}$ splits, and we can view $\pi_1\otimes \cdots\otimes \pi_t$ as a sub-representation of $\mathbf{r}_\alpha(\pi_1\times\cdots\times \pi_t)$.

Now, let $\phi$ be a morphism in $\Hom( \pi_1\times\cdots\times \pi_t, \sigma_1\times\cdots\times \sigma_t)$. By considerations of supercuspidal support, the image of $\psi:= \mathbf{r}_{\alpha}(\phi)|_{\pi_1\otimes\cdots\otimes\pi_t}$ is contained in $\sigma_1\otimes\cdots \otimes \sigma_t$ (after running similar arguments on $\epsilon_{\underline{\sigma}}$). Thus, we obtain the following commutative diagram

\[
\begin{diagram}
\node{\mathbf{r}_\alpha(\sigma_1\times\cdots\times \sigma_t)} \arrow{e,t}{\epsilon_{\underline{\pi}}} \node{\sigma_1\otimes \cdots\otimes \sigma_t} \\
\node{\mathbf{r}_\alpha(\pi_1\times\cdots\times \pi_t)} \arrow{n,t}{\mathbf{r}_{\alpha}(\phi)}\arrow{e,t}{\epsilon_{\underline{\sigma}}} \node{\pi_1\otimes \cdots\otimes \pi_t}\arrow{n,t}{\psi}
\end{diagram}\;,
\]
which implies by adjunction of functors the commutativity of the diagram
\[
\begin{diagram}
\node{\sigma_1\times\cdots\times \sigma_t} \arrow{e,t}{id} \node{\sigma_1\times \cdots\times \sigma_t} \\
\node{\pi_1\times\cdots\times \pi_t} \arrow{n,t}{\phi}\arrow{e,t}{id} \node{\pi_1\times \cdots\times \pi_t}\arrow{n,t}{\mathbf{i}_{\alpha}(\psi)}
\end{diagram}\;.
\]


\end{proof}

\begin{corollary}\label{disjsupp-cor}
Suppose that $\pi_i, \sigma_i \in \mathfrak{R}(G_{n_i})$ are given for $i=1,\ldots,t$, together with pairwise disjoint multisets $I_1,\ldots,I_t\in \mathbb{N}(\mathcal{C})$, such that for every irreducible subquotient $\tau$ of either $\pi_i$ or $\sigma_i$, $\supp(\tau) \subseteq I_i$ holds.

Then, we have a natural identification

\[
\Hom( \pi_1\times\cdots\times \pi_t, \sigma_1\times\cdots\times \sigma_t) \cong \Hom(\pi_1,\sigma_1) \otimes\cdots\otimes \Hom (\pi_t ,\sigma_t) \;.
\]
\end{corollary}

\begin{proof}
It follows directly by decomposing each $\pi_i, \sigma_i$ into direct sums according to supercuspidal support, and applying Proposition \ref{disjsupp} on each of the summands.

\end{proof}

\subsection{Tadic classification of the unitary spectrum}\label{sect-tadic}
Let $\irr^u\subset \irr$ be the subset of irreducible representations that are unitarizable, that is, those whose space can be equipped with a positive definite Hermitian form invariant under the group action.

A classification of $\irr^u$ in terms of multisegments was achieved by Tadic in \cite{Tad}. Let us briefly recall this classification.

It is easy to check that $\irr^u \cap\, \mathcal{C} = \{\rho\in \mathcal{C}\;:\; \Re(\rho)=0\}$.

For each pair of integers $a,b\in \mathbb{Z}_{>0}$ and a supercuspidal $\rho\in\irr^u \cap\, \mathcal{C}$, we define the \textit{Speh} multisegment
\[
\gotM^{a,b}_\rho = \sum_{i=1}^b \left[ \frac{b-a}2+1-i, \frac{b+a}2-i\right]_\rho\in \multi\;.
\]
We then set $\pi^{a,b}_\rho:= L(\gotM^{a,b}_\rho)$ to be a (unitary\footnote{It is also common to define a general Speh representation without the requirement that $\rho$ is unitarizable.}) \textit{Speh representation}.

Let $B\subset \multi$ denote the collection of Speh multisegments. Let us also define
\[
B^{comp} = \{ \gotM\nu^{\alpha} + \gotM\nu^{-\alpha}\;:\; \gotM\in B,\; 0<\alpha<1/2\}\;.
\]

\begin{theorem}\label{tadic}
  Let $U\subset \multi$ be the sub-monoid generated by $B \sqcup B^{comp}$ in $\multi$. Then,
  \[
  L(U) = Z(U) = \irr^u\;.
  \]
\end{theorem}

It is known (\cite{bern}) that $\pi_1\times\pi_2$ is irreducible, for any $\pi_1,\pi_2\in \irr^u$. Hence, the theorem above states that any $\pi\in \irr^u$ can be written in the form $\pi = \pi_1\times \cdots \times\pi_t$, where $\pi_i = L(\gotM_i)$ with $\gotM_i\in B\sqcup B^{comp}$, for all $i=1,\ldots,t$.

\begin{definition}\footnote{There may be an ambiguity regarding this definition, when compared to other sources. Arthur-type representations were defined with the goal of parameterizing the local components of the discrete automorphic spectrum of a group. Assuming the Ramanujan conjecture for $GL_n$, our definition should coincide with this global notion.}
    We say that $\pi\in \irr^u$ is of \textit{Arthur-type}, if $\gotM_i\in B$, for all $i=1,\ldots,t$ in the decomposition above.
\end{definition}

  In other words, Arthur-type representations are products of Speh representations.
\begin{definition}
We say that $\pi\in\irr^u$ is of \textit{proper Arthur-type}, if $\pi\cong \pi^{a_1,b_1}_\rho\times\cdots \times \pi^{a_k,b_k}_\rho$, for a fixed $\rho\in \irr^u \cap\, \mathcal{C}$ and integers $a_i,b_i\in \mathbb{Z}_{>0}$.

\end{definition}

Note, that proper Arthur-type representations built out of non-isomorphic elements of $\irr^u \cap\, \mathcal{C}$ will always have disjoint supercuspidal supports.

Given an Arthur-type representation $\pi = L(\gotM)$ and a real number $0<\alpha<1/2$, let us write
\[
\pi(\alpha):= \pi\nu^\alpha\times \pi\nu^{-\alpha} = L( \gotM\nu^{\alpha} + \gotM\nu^{-\alpha})\in \irr^u\;.
\]

It is easy to deduce from Theorem \ref{tadic} that for every $\pi\in \irr^u$ there is a unique, up to order, factorization of the form
\[
\pi \cong \pi_0 \times \pi_1(\alpha_1)\times\cdots\times \pi_k(\alpha_k)\;,
\]
where $\pi_0,\pi_1,\ldots,\pi_k\in \irr^u$ are of Arthur-type\footnote{At least for the purposes of this decomposition, we need to consider $L(0)\in \irr(G_0)$ (treated as neutral to multiplication) as an Arthur-type representation.} and $0<\alpha_1, \ldots, \alpha_k< 1/2$ are distinct real numbers.

\subsection{Ladder representations and the Lapid-M\'{i}nguez criterion}\label{sect-lm}

Speh representations are a special case of a class of irreducible representations known as \textit{ladder} representations. A representation $\pi\in \irr$ is called a \textit{proper}\footnote{For ease of exposition, we do not define a general ladder here. The ladder representations that occur in this manuscript are all proper ladder.}\textit{ ladder} representation, if it is given as $\pi = L(\Delta_1 + \ldots + \Delta_k)$, for segments $\Delta_i,\,i=1,\ldots,k$, satisfying
\[
\Delta_k \prec  \ldots\prec\Delta_2 \prec \Delta_1\;.
\]
Ladder representations were shown (e.g. \cite{LM,LM2,me-decomp}) to possess certain remarkable properties, which often make the ladder class more approachable for treatment of questions on general irreducible representations. Essentially the same class of representations was also studied in the literature under various names in various type $A$ settings, such as calibrated affine Hecke algebra modules, snake modules for quantum affine algebras and homogeneous modules for KLR algebras.

Given a ladder representation $\pi\in \irr$ and any representation $\sigma\in \irr$, it was shown \cite[5.15]{LM2} that both $\pi\times \sigma$ and $\sigma\times \pi$ have a unique irreducible sub-representation. Lapid-M\'{i}nguez have also devised an algorithm in \cite{LM2} for computing the multisegment classifying that sub-representation. We recall in what follows one corollary of that algorithm.

Suppose, for that purpose, that
\[
\gotM_1 = \Delta_1 + \ldots + \Delta_{k_1},\quad \gotM_2 = \Delta'_1 + \ldots + \Delta'_{k_2}\;,
\]
are two multisegments with $\Delta_{k_1} \prec \ldots \prec \Delta_1$ and $\Delta'_{k_2} \prec \ldots \prec \Delta'_1$.

Recall again that $Z(\gotM_1 + \gotM_2)$ always appears as a sub-quotient in $Z(\gotM_1)\times Z(\gotM_2)$. Among the results of \cite{LM2} is a combinatorial criterion for determining when does $Z(\gotM_1 + \gotM_2)$ actually appear as a sub-representation of $Z(\gotM_1)\times Z(\gotM_2)$.

Consider the set of indices $I = \{1,\ldots, k_1\}\times \{1,\ldots, k_2\}$, and the following bipartite graph on the set of vertices $I_1\sqcup I_2$, where $I_1= I_2 = I$ (two copies of $I$).  We say that elements $(i_1,j_1)\in I_1$ and $(i_2,j_2)\in I_2$ are in relation $(i_1,j_1) \leftrightarrow (i_2,j_2)$, if either $i_1 = i_2$ and $j_2 = j_1+1$, or $j_1= j_2$ and $i_2 = i_1-1$.

Consider the sets
\[
X_{\gotM_1;\gotM_2} = \{ (i,j)\in I_1\;:\; \Delta_i \prec \Delta'_j\}\;, \quad
Y_{\gotM_1;\gotM_2} = \{ (i,j)\in I_2\;:\; \Delta_i \prec \overrightarrow {\Delta}'_j\}\;.
\]
Here $\overrightarrow{\Delta}$ means the segment $[s+1,t+1]_\rho$, for a segment $\Delta = [s,t]_\rho$.

We can consider the bipartite graph $(X_{\gotM_1;\gotM_2},Y_{\gotM_1;\gotM_2}, \leftrightarrow)$ created by restricting the relation $\leftrightarrow$ onto edges between $X_{\gotM_1;\gotM_2}$ (as a subset of $I_1$) and $Y_{\gotM_1;\gotM_2}$ (as a subset of $I_2$).

Recall that a \textit{matching function} on a bipartite graph whose two parts of vertices are described by sets $X,Y$, is a function $f:X\to Y$, such that there is an edge between $x$ and $f(x)$, for all $x\in X$.

\begin{theorem}\cite[5.21]{LM2}\label{lap-min}
The unique irreducible sub-representation of $Z(\gotM_1)\times Z(\gotM_2)$ is isomorphic to $Z(\gotM_1 + \gotM_2)$, if and only if, there is a matching function for the bipartite graph $(X_{\gotM_1;\gotM_2},Y_{\gotM_1;\gotM_2}, \leftrightarrow)$, i.e., there exists an injective function
\[
f:X_{\gotM_1;\gotM_2}\to Y_{\gotM_1;\gotM_2}\;,
\]
which satisfies $x\leftrightarrow f(x)$, for all $x\in X_{\gotM_1;\gotM_2}$.

\end{theorem}

\subsection{Bernstein-Zelevinski derivatives}
For given $n_1,n_2$, consider the subgroup $U< G_{n_1}\times G_{n_2}$ of upper unitriangular matrices in $G_{n_2}$. Let $\psi$ be a non degenerate character of $U$.

We then have an obvious functor of taking co-invariants:
\[
\begin{array}{cccc} W_{n_1,n_2}: & \mathfrak{R}(G_{n_1}\times G_{n_2}) & \to & \mathfrak{R}(G_{n_1}) \\
            & V & \mapsto & V/\sspan\{ g\cdot v - \psi(g)v\;:\; g\in G_{n_2},\,v\in V\}
            \end{array}\;.
\]
The Bernstein-Zelevinski derivatives of \cite{BZ1} can be defined as functors $\mathfrak{R}(G_{n}) \to \mathfrak{R}(G_{n-i})$ constructed by composing $W$ with the Jacquet functor.

More precisely, given $\pi \in \mathfrak{R}(G_n)$, we set its $i$-th derivative to be
\[
\pi^{(i)}: = W_{n-i,i}(\mathbf{r}_{(n-i,i)}(\pi))\in \mathfrak{R}(G_{n-i})\;,
\]
for all $1\leq i\leq n$. Here $\mathbf{r}_{(n-i,i)}$ is the Jacquet functor left-adjoint to $\mathbf{i}_{(n-i,i)}$.

We set $\pi^{(0)} = \pi$.

The derivatives comply with a Leibniz rule, in the following sense.

\begin{proposition}\label{leibn}
Suppose that $\pi_i\in \mathfrak{R}(G_{n_i})$, $i=1,\ldots,k$ are given.
\begin{enumerate}
  \item  The representation $(\pi_1\times\cdots\times \pi_k)^{(j)}$ has a filtration whose constituents are given by all representations of the form $\pi^{(s_1)}_1\times \cdots\times \pi^{(s_k)}_k$, with $0\leq s_i\leq n_i$ for all $i=1,\ldots, k$ and $j= s_1+\ldots + s_k$.
  \item Suppose that $\pi_i^{(m_i)}$ is the highest non-zero derivative of $\pi_i$, for all $i=1,\ldots, k$. Suppose that $j = m_t+ m_{t+1}+\ldots + m_k$, for some $1\leq t\leq k$. Then, the representation
      \[
      \pi_1\times \cdots \times \pi_{t-1}\times \pi_t^{(m_t)}\times \cdots\times \pi_k^{(m_k)}
      \]
  is a quotient of $(\pi_1\times\cdots\times \pi_k)^{(j)}$, that is, the uppermost piece of the above filtration.
\end{enumerate}
  \begin{proof}
  \cite[Corollary 4.14(c)]{BZ1} together with its proof.
  \end{proof}

\end{proposition}

\section{Affine Hecke algebras}\label{sect-hecke}
Given $n\in \mathbb{Z}_{>0}$ and $q\in \mathbb{C}$ (which for our needs will be assumed to be non-root of unity), the root datum of $GL_n$ gives rise to the (extended) \textit{affine Hecke algebra} $H(n,q)$. In fact, we will not be using here the concrete algebraic structure of these algebras, but rather some abstract information on their categories of representations.

Yet, to avoid confusion let us recall a possible presentation of $H(n,q)$: This is the complex algebra generated by $T_1,\ldots, T_{n-1}$ and invertible $y_1,\ldots,y_n$, subject to the relations
\[
\begin{array}{ll}
T_i T_{i+1} T_i = T_{i+1} T_i T_{i+1},\; & \forall 1\leq i\leq n-2\\

(T_i -q)(T_i+1)=0,\;& \forall 1\leq i \leq n-1\\

T_iT_j = T_j T_i,\;  & \forall |j-i|>1\\

y_iy_j = y_jy_i,\;& \forall 1\leq i,j\leq n\\

T_i y_iT_i = qy_{i+1},\; &\forall 1\leq i\leq n-1\\

T_i y_j = y_jT_i,\; &\forall j\neq i, i+1\;.
\end{array}\;
\]

We denote by $\mathcal{M}^q_n$ the category of finite-dimensional modules over the algebra $H(n,q)$.

For any $n_1,\ldots, n_t\in\mathbb{Z}_{>0}$, there is a natural embedding of algebras
\[
H(n_1,q)\otimes \cdots \otimes H(n_t,q)\hookrightarrow H(n_1+\ldots+ n_t,q)\;.
\]
This embedding gives rise to an induction functor
\[
\mathbf{i}_{(n_1,\ldots, n_t)}: \mathcal{M}^q_{n_1} \times \cdots \times  \mathcal{M}^q_{n_t} \to \mathcal{M}^q_{n_1+\ldots+n_t}\;,
\]
which we will simply denote as a product operation, i.e., $\pi_1\times \pi_2:= \mathbf{i}_{(n_1,n_2)}(\pi_1\otimes \pi_2)$, for $\pi_i \in \mathcal{M}^q_{n_i}$.

\subsection{Equivalence to Bernstein blocks}\label{sect-equiv-heck}

Let $\sim$ be an equivalence relation on the elements of $\mathbb{N}(\mathcal{C})$ defined as follows: For $A,B\in \mathbb{N}(\mathcal{C})$, we say that $A\sim B$, if there are representations $\rho_1,\ldots,\rho_t\in \mathcal{C}$ and numbers $s_1,\ldots,s_t\in \mathbb{C}$, such that
\[
A = \rho_1 +\ldots + \rho_t,\quad B = \rho_1\nu^{s_1} + \ldots +\rho_t\nu^{s_t}\;.
\]
Note, that the number $N_A: = n_1 + \ldots + n_t$, where $\rho_i \in \mathcal{C}(G_{n_i})$, for $i=1,\ldots,t$, is an invariant of the $\sim$-equivalence class of $A$ (also known as the inertia class). Hence, we will write $N_\Theta$, where $\Theta$ denotes that equivalence class. 

Each inertia class $\Theta$ defines the \textit{Bernstein block}\footnote{These sub-categories may become blocks in the more axiomatic treatment of Abelian cateogries when dealing with larger categories of all smooth representations of a group.}  $\mathfrak{R}(\Theta)$, which is the full sub-category of $\mathfrak{R}(G_{N_\Theta})$ consisting of representations $\pi$, such that $\supp(\sigma)$ belongs to $\Theta$ for all irreducible sub-quotients of $\pi$.

The Bernstein decomposition \cite{bern-center} (in the case of $GL_n$) states that we have a decomposition of Abelian categories
\[
\mathfrak{R}(G_n) = \prod_{\Theta} \mathfrak{R}(\Theta)\;,
\]
where the product goes over all inertia classes $\Theta$, for which $N_\Theta = n$.

When writing such a decomposition, we mean that every object $M$ in $\mathfrak{R}(G_n)$ is decomposed uniquely as $M\cong \oplus_\Theta M_\Theta$, where each $M_\Theta$ is an object in $\mathfrak{R}(\Theta)$, all but finitely many $M_\Theta$ are zero (we deal with finite length objects), and
\[
\Hom_{\mathfrak{R}(G_n)} = \oplus_\Theta \Hom_{\mathfrak{R}(\Theta)}( M_\Theta, N_\Theta)\;,
\]
holds, for all $M,N$ in $\mathfrak{R}(G_n)$.

We will call $\mathfrak{R}(\Theta)$ a \textit{simple block}, if $\Theta = \Theta(\rho,d)$ has a representative of the form $d\cdot\rho\in \mathbb{N}(\mathcal{C})$, where $\rho\in \mathcal{C}$ and $d\geq1$ is an integer. 

It is evident that for $\pi_1 \in \mathfrak{R}(\Theta(\rho,n_1))$ and $\pi_2 \in \mathfrak{R}(\Theta(\rho,n_2))$, the representation $\pi_1\times \pi_2$ belongs to the block $\mathfrak{R}(\Theta(\rho,n_1+n_2))$.

The clear consequence of the above is that for any $\rho\in \mathcal{C}$, the irreducible representations appearing in the sequence of blocks $\{\mathfrak{R}(\Theta(\rho,n))\}_{n=0}^\infty$ are precisely those given by $Z(\multi_\rho)$ or $L(\multi_\rho)$.

For the trivial representation $\nu^0$ of $G_1$, we set $\Theta_n= \Theta(\nu^0,n)$, for all $n\geq1$. The simple block $\mathfrak{R}(\Theta_n)$ is called the \textit{principal} (or \textit{Iwahori-invariant}) block of $\mathfrak{R}(G_n)$.

It was shown in \cite{bk-ss} and \cite{heier-cat} that for every simple block $\Theta=\Theta(\rho,d)$, we have an equivalence of categories
\[
U_\Theta :  \mathfrak{R}(\Theta) \xrightarrow{\sim}\mathcal{M}^{q_F^{o(\rho)}}_{d}\;.
\]

The equivalences $\{U_\Theta\}$ are not canonical, yet they can be chosen in a way that is compatible with parabolic induction \cite{roche}. Namely, we are allowed to assume that
\begin{equation}\label{monoidU}
U_{\Theta(\rho,n_1+n_2)}(\pi_1\times\pi_2) \cong  U_{\Theta(\rho,n_1)}(\pi_1) \times U_{\Theta(\rho,n_2)}(\pi_2)\;
\end{equation}
holds, for all $\pi_1 \in \mathfrak{R}(\Theta(\rho,n_1))$ and $\pi_2 \in \mathfrak{R}(\Theta(\rho,n_2))$.

The particular case of equivalences for principal blocks
\[
U_n= U_{\Theta_n}:  \mathfrak{R}(\Theta_n)\xrightarrow{\sim} \mathcal{M}^{q_F}_n \, ,\;\forall n\;
\]
is in fact a classical theorem\footnote{More precisely, the cited theorem relates the principal block to representations of the Iwahori-Hecke algebra of $G_n$, which is then known by a result of Bernstein to be isomorphic to our definition of $H(n,q)$ (see, for example, the lectures \cite{howe-heck}).} of Borel \cite{borel-equiv} and Casselman.

We will fix a canonical choice of $\{U_n\}$ which is supplied by said theorem.

Let us note that the irreducible representations of the algebra $H(1,q_F)=\mathbb{C}[y_1,y_1^{-1}]$ are naturally given by the variety $\mathbb{C}^\times$. Note further that the irreducible representations in $\mathfrak{R}(\Theta_1)$ are given by the group of unramified characters $\{\nu^s\,:\,s\in\mathbb{C}\}$. Under these identifications, the canonical equivalence $U_1$ takes $\nu^s$ to the character given by $q_F^s\in\mathbb{C}^\times$.

We can use $U_n$ to push a parametrization of the irreducible representations of $H(n,q_F)$ in terms of multisegments. Recall that the irreducible representations in $\mathfrak{R}(\Theta_n)$ are parameterized by $\multi_0$, either through $Z$ or $L$. Hence, we write bijections
\[
\hat{Z}, \hat{L} : \multi_0 \to \bigcup_{n\geq0}\irr(\mathcal{M}^{q_F}_n) \;,
\]
defined by $\hat{Z}(\gotM) = U_{n_\gotM}(Z(\gotM))$, where $Z(\gotM)\in \irr(G_{n_\gotM})$. Similarly, $\hat{L}$ is defined by $L$.

\begin{remark}\label{rem-int}
Note, that $\hat{Z}$ (or $\hat{L}$) can in fact be described intrinsically (that is, without use of $p$-adic groups) as was done in \cite{rog-hecke}. In particular, these classifications are independent of the field $F$.
\end{remark}

The following proposition shows that all equivalences $\{U_\Theta\}$ preserve the Zelevinski and Langlands parametrizations of irreducible representations, in a natural sense.

\begin{proposition}\label{equiv-block}
Let $\rho\in \mathcal{C}$ be given.

Let $E$ be a $p$-adic field with residue cardinality $q_E = q_F^{o(\rho)}$. Let
\[
\phi = \phi_{\nu^0_E, \rho}:   \multi_0^E \to \multi_\rho
\]
be the isomorphism defined in Section \ref{sect-multis}, where $\nu^0_E$ stands for the trivial representation of $GL_1(E)$.

Let us write
\[
\hat{Z} : \multi^E_0 \to \bigcup_{n\geq0}\irr(\mathcal{M}^{q_E}_n) \;,
\]
for the map as defined above (but for the field $E$ in place of $F$).

Then, we can choose the collection $\{U_{\Theta(\rho,n)}\}_n$, so that it satisfies
\[
\hat{Z}(\gotM) \cong U_{\Theta(\rho, n_\gotM)}(Z(\phi(\gotM)))\;,
\]
for all $\gotM\in \multi^E_0$, where $\hat{Z}(\gotM)\in \mathcal{M}^{q_E}_{n_\gotM}$.

\end{proposition}

\begin{proof}
Recall that for a multisegment $\gotM = \sum_{i=1}^k \Delta_i \in \multi_0^E$, $Z(\gotM)$ is defined to be the unique irreducible sub-representation of $\zeta(\gotM)= Z(\Delta_1)\times\cdots \times Z(\Delta_k)$ (for a prescribed ordering of the segments of $\gotM$, see Section \ref{sect-multis}). Furthermore, each $Z(\Delta_i)$ is the unique irreducible sub-representation of $\nu_E^{s^i_1}\times \cdots\times \nu_E^{s^i_{n_i}}$, for some numbers $s^i_1,\ldots,s^i_{n_i}\in \mathbb{C}$.

Similarly, $Z(\phi(\gotM))$ is constructed by the same process, while replacing the role of each $\nu_E^{s^i_j}$ by $\rho\nu^{s^i_j}$.

Hence, using the compatibility property \eqref{monoidU}, the fact that $\{U_\Theta\}$ are all equivalences of abelian categories and Remark \ref{rem-int}, it is enough to check the statement for $U_{\Theta(\rho,1)}$.

In other words, we need to verify that $U_{\Theta(\rho,1)}$ is allowed to be chosen so that for all $s\in\mathbb{C}$, the representation $\rho\nu^s$ in $\mathfrak{R}(\Theta(\rho,1))$ is sent to the character of $H(1, q_E)=\mathbb{C}[y_1,y_1^{-1}]$ given by $y_1\mapsto q_E^s$.

Let us recall the construction of $U_{\Theta(\rho,1)}$ in \cite{heier-cat}. Let $V$ be the space of the representation $\rho\in \mathcal{C}(G_m)$. Then, $G_m$ acts on the space $V[t,t^{-1}]$ by
\[
\overline{\rho}(g)\cdot(vt^k) = \rho(g)vt^{k+ val(\det(g))},\;\forall g\in G_m,\,v\in V,\,k\in \mathbb{Z}\;,
\]
where $|a|_F = q_F^{val(a)}$, for $a\in F$.

The natural action of the ring of Laurent polynomials $\mathbb{C}[t,t^{- 1}]$ on $V[t,t^{-1}]$ intertwines the $\overline{\rho}$ action of $G_m$. The construction in \cite{heier-cat} specifies a sub-representation $W< V[t,t^{-1}]$ of $\overline{\rho}$ which is stable under the action of the sub-algebra $H = \mathbb{C}[t^{o(\rho)},t^{-o(\rho)}]\subset \mathbb{C}[t,t^{-1}]$. Moreover, we have
\[
End(\overline{\rho}|_W) = H\;.
\]

The functor $U_{\Theta(\rho,1)}$ is then defined by taking $\pi\in\mathfrak{R}(\Theta(\rho,1))$ to
\[
U_{\Theta(\rho,1)}(\pi):= \Hom_{G_m} (W, \pi)\;,
\]
viewed as a $H\cong H(1,q_E)$-module.

For all $s\in\mathbb{C}$, there is a projection of representations $\psi_s: \overline{\rho}\to \rho\nu^s$ given by $\psi_s(vt^k) = q_F^{ks}v$. It is easy to verify that $U_{\Theta(\rho,1)}(\rho\nu^s)$ is a one-dimensional space spanned by $\psi_s|_W$.

Note, that for the generator $t^{o(\rho)}\in H$ and $w\in W$, we have $\psi_s(t^{o(\rho)}\cdot w) = q_F^{o(\rho)s}\psi_s(w)$. Hence, $H(1,q_E)$ acts on $U_{\Theta(\rho,1)}(\rho\nu^s)$ through the character $y_1\mapsto q_E^s$.

\end{proof}

\subsection{Consequences of a result of Hernandez}\label{sect-conseq}

Section \ref{sect-qa} deals with quantum affine algebras and the quantum affine Schur-Weyl duality functor. The following theorem on representations of affine Hecke algebras will be shown to be a manifestation of a theorem of Hernandez \cite{hernan-cyc}, when transferred through the duality functor.

\begin{theorem}\label{transl}
  Let $\pi_1 = \hat{Z}(\gotM_1),\ldots,\pi_t = \hat{Z}(\gotM_t)$ be irreducible representations in $\mathcal{M}_{k_1}^{q_F},\ldots,\mathcal{M}_{k_t}^{q_F}$, respectively.

If $\pi_i\times \pi_j$ has a unique irreducible quotient which is parameterized by $\hat{Z}(\gotM_i+\gotM_j)$, for all $1\leq i < j\leq t$, then the representation $\pi_1\times\cdots\times \pi_t$ has a unique irreducible quotient which is parameterized by $\hat{Z}(\gotM_1 +\cdots + \gotM_t)$.

\end{theorem}

We would like to extend the statement of the theorem above slightly.

For that purpose let us recall that each algebra $H(n,q)$ is equipped with the \textit{Iwahori-Matsumoto involutive automorphism} $\theta_n$ (see \cite[I.6]{mw-zel}). It gives rise to an involutive auto-equivalence of $\mathcal{M}_n^q$. In order to ease notation, we will simply write $\theta$ for all these involutions.

It is known \cite[Lemme I.7.1]{mw-zel}\footnote{There is an obvious typo in the statement of the lemma in that source.} that the Iwahori-Matsumoto involution is compatible with the induction product, in the following sense. For all representations $\pi_1,\ldots,\pi_t$ in   $\mathcal{M}_{k_1}^q,\ldots,\mathcal{M}_{k_t}^q$, respectively, we have
\[
\theta(\pi_1\times\cdots\times \pi_t)\cong \theta(\pi_t)\times \cdots\times \theta(\pi_1)\;.
\]

When restricting $\theta$ to irreducible representations, we obtain \cite[Proposition I.7.3]{mw-zel} what is known as the \textit{Zelevinski involution} in the context of $p$-adic groups, that is,
\[
\theta(\hat{Z}(\gotM))\cong \hat{L}(\gotM)\;,
\]
for all $\gotM \in\multi_0$.

\begin{corollary}\label{cor-generalhern}
The statement of Theorem \ref{transl} remains valid, when $\hat{Z}$ is replaced with $\hat{L}$. In addition, ``quotient" may be replaced with ``sub-representation".
\end{corollary}
\begin{proof}
It follows from an application of the functor $\theta$ and Proposition \ref{GK}, which remains valid for representations of affine Hecke algebras through the equivalences $\{U_n\}$.
\end{proof}

\section{Classes of irreducible representations}\label{sect-work}

We would like to study certain classes of representations in $\irr$ which naturally occur in the derivatives of unitarizable representations.

\subsection{Quasi-Speh representations}

We will first deal with a subclass of ladder representations, which we will call \textit{quasi-Speh} representations. These are parameterized by integers $a,b\in \mathbb{Z}_{>0},\, c\in \mathbb{Z}$, so that $0\leq c\leq a$, and a representation $\rho\in \irr^u\,\cap\, \mathcal{C}$. For such data, we define the multisegment
\[
\gotM_\rho^{a,b,c} = \left[ \frac{b+a}2-c, \frac{b+a}2-1\right]_\rho + \sum_{i=2}^b \left[ \frac{b-a}2+1-i, \frac{b+a}2-i\right]_\rho\in \multi_\rho\;.
\]
We then set $\pi_\rho^{a,b,c}:= L(\gotM^{a,b,c}_\rho)$ to be a quasi-Speh representation.

Note, that for $a=c$, these are the usual Speh representations $\pi^{a,b}_\rho = \pi^{a,b,a}_\rho $ defined in Section \ref{sect-tadic}. We also note that for $b>1$ and all $a$, we have $\pi_\rho^{a,b,0} = \pi_\rho^{a,b-1}\nu^{-1/2}$.

The following identities make the class of quasi-Speh representations relevant to our discussion. See \cite[Section 5.4]{LM} for the proof, which is attributed to Tadic.

\begin{proposition}\label{deriv-speh}
Let $\rho\in \irr^u\,\cap\, \mathcal{C}$ be a given representations of $G_d$. Let $a,b\in \mathbb{Z}_{>0}$ be given.

Then, the formula
\[
(\pi_{\rho}^{a,b})^{(i)} \cong \left\{\begin{array}{cc}
                                        \pi_{\rho}^{a,b,a-k} & i = kd,\;0\leq k\leq a \\
                                         0 & \mbox{\small{otherwise}}
                                      \end{array} \right.
\]
gives the derivatives of the Speh representation $\pi^{a,b}_{\rho}$.
\end{proposition}

Let us define a preorder on the class of quasi-Speh representations. We will write
\[
\pi_\rho^{a_1,b_1,c_1} \preceq \pi_\rho^{a_2,b_2,c_2}\;,
\]
if $a_1+b_1 < a_2+b_2$, or if $a_1+b_1 = a_2+b_2$ and $a_1\leq a_2$.

For any given $\rho\in\irr^u\,\cap\, \mathcal{C}$, the restriction of $\preceq$ to the collection $\{\pi^{a,b,c}_\rho\}_{a,b,c}$ gives a total preorder.

\begin{proposition}\label{replace}
If $\pi_1= L(\gotM_1),\pi_2= L(\gotM_2)$ are two quasi-Speh representations, which satisfy $\pi_2\preceq \pi_1$, then $\pi_1\times \pi_2$ has a unique irreducible quotient which is given by $L(\gotM_1+\gotM_2)$.

\end{proposition}

\begin{proof}
Suppose that $\pi_1 = \pi_\rho^{a_1,b_1,c_1}$ and $\pi_2 = \pi_\rho^{a_2,b_2,c_2}$. Let us denote
\[
\gotM_1 = \gotM_\rho^{a_1,b_1,c_1} = \Delta_1 + \ldots + \Delta_{k_1},\quad \gotM_2 = \gotM_\rho^{a_2,b_2,c_2} = \Delta'_1 + \ldots + \Delta'_{k_2}\;,
\]
so that $\Delta_{k_1} \prec \ldots \prec \Delta_1$ and $\Delta'_{k_2} \prec \ldots \prec \Delta'_1$ are the defining segments. When $c_1$ (resp. $c_2$) equal to zero, the condition $\Delta_2 \prec \Delta_1$ (resp. $\Delta'_2 \prec \Delta'_1$) is waived, but we will still write $\Delta_1$ (resp. $\Delta'_1$), while referring to the empty segment.

By same considerations as in the proof of Corollary \ref{cor-generalhern}, it is enough to prove that $Z(\gotM_1 + \gotM_2)$ is the socle of $Z(\gotM_1)\times Z(\gotM_2)$. To show that, we will apply the Lapid-M\'{i}nguez criterion of Theorem \ref{lap-min}.

Recall the sets $X=X_{\gotM_1;\gotM_2}, Y=Y_{\gotM_1;\gotM_2}$ and the relation $\leftrightarrow$, as they were defined in Section \ref{sect-lm}.

We need to show there is an injective function $f:X\to Y$, which satisfies $x\leftrightarrow f(x)$, for all $x\in X$.

Let us write the set of indices
\[
K = \{(i,i-1)\in [1,\ldots,k_1]\times [1,\ldots,k_2]\}\;.
\]
Note, that if $X\cap K\neq \emptyset$, we have $e_\rho(\Delta_{i_1})< e_\rho(\Delta'_{i_1-1}) = e_\rho(\Delta'_{i_1})+1$ for some index $i_1$. It follows that $b_1+a_1\leq b_2+a_2$. Yet, since $\pi_2\preceq \pi_1$ holds, we must have $b_1+a_1= b_2+a_2$ and $a_2 \leq a_1$. This implies $b_1\leq b_2$ and that $(i,i)\in Y$ holds, for all $2\leq i\leq k_1$.

For every $(i,i-1)\in X\cap K$, let us set $f(i,i-1) = (i,i)$.

Now, let us consider the Speh representations given as $\hat{\pi}_1 = \pi_\rho^{a_1,b_1}$ and $\hat{\pi}_2 = \pi_\rho^{a_2,b_2}$. Then,
\[
\hat{\pi}_1 =L(\hat{\gotM}_1)= L\left( \sum_{i=1}^{k_1} \hat{\Delta}_i\right) ,\quad \hat{\pi}_2 =  L(\hat{\gotM}_2)=L\left( \sum_{i=1}^{k_2} \hat{\Delta}'_i\right)\;,
\]
where $\hat{\Delta}_i = \Delta_i$ and $\hat{\Delta}'_i = \Delta'_i$, for all $2\leq i$, while $b_\rho(\Delta_1)\geq  b_\rho(\hat{\Delta}_1)$, $b_\rho(\Delta'_1)\geq b_\rho(\hat{\Delta}'_1)$, $e(\Delta_1)=e(\hat{\Delta}_1)$ and $e(\Delta'_1)=e(\hat{\Delta}'_1)$.

It is known that $\hat{\pi}_1\times \hat{\pi}_2$ is irreducible, as a product of unitarizable irreducible representations. Hence, we can apply Theorem \ref{lap-min} to obtain a matching function $g:\hat{X}\to \hat{Y}$ for the restricted bipartite graph $(\hat{X} = X_{\hat{\gotM}_1;\gotM_2}, \hat{Y} = Y_{\hat{\gotM}_1;\hat{\gotM}_2}, \leftrightarrow)$.

Note, that since  $\pi_2\preceq \pi_1$, we know that $e_\rho(\Delta_1) = (b_1+a_1)/2-1$ is no smaller than $e_\rho(\Delta'_j)$, for all $1\leq j\leq k_2$. Thus, $X, \hat{X}\subset [2,\ldots, k_1]\times [1,\ldots, k_2]$.

Clearly, we also have
\[
X\cap \left([2,\ldots, k_1]\times [2,\ldots, k_2]\right) = \hat{X} \cap \left([2,\ldots, k_1]\times [2,\ldots, k_2]\right)\;,
\]
\[
Y\cap \left([2,\ldots, k_1]\times [2,\ldots, k_2]\right) = \hat{Y} \cap \left([2,\ldots, k_1]\times [2,\ldots, k_2]\right)\;.
\]

Suppose that $(i,j)\in (X\cap \hat{X})\setminus K$, and denote $(i',j') = g(i,j)\in \hat{Y}$. In case $j'= j+1$, we see that $i',j'\geq 2$, which implies that $(i',j')\in Y$. Otherwise, $(i',j') = (i-1, j)$. In case that $i>2$, we have $\Delta_{i-1} = \overrightarrow{\Delta_i} \prec \overrightarrow{\Delta'_j}$, since $(i,j)\in X$. Thus, again we have $(i',j')\in Y$.

As for the case that $i=2$, the inclusion $(2,j)\in X$ would have implied that
\[
e_\rho(\Delta'_1)\geq e_\rho(\Delta'_j)\geq e_\rho(\Delta_2)+1 = \frac{b_1+a_1}2 -1\;.
\]
On the other hand, it follows from $\pi_2\preceq \pi_1$ that $e_\rho(\Delta'_1) \leq \frac{b_1+a_1}2 -1$. Hence, $j=1$ and $(i,j)\in K$ would have given a contradiction.



Having established that $g((X\cap \hat{X})\setminus K)$ is always contained in $Y$, we can define $f= g$ on $(X\cap \hat{X})\setminus K$. Injectivity is not interfered with the definition of $f|_K$, because when $x\in X$ satisfies $x\leftrightarrow (i,i)$ for an index $i$, we clearly must have $x\in K$.

We are left with the task of extending $f$ injectively to
\[
X\setminus(\hat{X}\cup K)\subset [3,\ldots,k_1]\times\{1\}\;.
\]

Suppose that $(i,1)\in X\setminus(\hat{X}\cup K)$. Since $i-1>1$, we again have $\Delta_{i-1} = \overrightarrow{\Delta_i} \prec \overrightarrow{\Delta'_1}$, which gives $(i-1,1)\in Y$. Moreover, the same argument shows that $(i,1)\not\in \hat{X}$ implies $(i-1,1)\not\in \hat{Y}$. Hence, we can extend $f$ by setting $f(i,1) = (i-1,1)$ without harming injectivity.


\end{proof}

\subsection{Quasi-Arthur-type representations}

We say that $\pi\in\irr$ is \textit{quasi-Arthur-type}, if it has the form
\[
\pi = L\left( \gotM_\rho^{a_1,b_1,c_1} + \ldots + \gotM_\rho^{a_k,b_k,c_k}\right)\;,
\]
for some integers $\{a_i,b_i,c_i\}_{i=1}^k$ and $\rho\in \irr^u\,\cap\, \mathcal{C}$.

The significance of quasi-Arthur-type representations appears through the following corollary of previous discussions.

\begin{proposition}\label{main-prop}
Let $\rho\in \irr^u\,\cap\, \mathcal{C}$ be fixed.
Suppose that
\[
\pi_\rho^{a_k,b_k,c_k}\preceq \ldots\preceq \pi_\rho^{a_1,b_1,c_1}
\]
are given quasi-Speh representations, for some integers $\{a_i,b_i,c_i\}$.

Then, the product
\[
\pi_\rho^{a_1,b_1,c_1}\times  \cdots \times \pi_\rho^{a_k,b_k,c_k}
\]
has a unique irreducible quotient, whose isomorphism class is given by the quasi-Arthur-type representation
\[
L\left( \gotM_\rho^{a_1,b_1,c_1} + \ldots + \gotM_\rho^{a_k,b_k,c_k}\right)\;.
\]
\end{proposition}

\begin{proof}
By Proposition \ref{equiv-block} and the property \eqref{monoidU}, we can assume that we are dealing with finite dimensional representations of the corresponding affine Hecke algebras. The statement then follows from Proposition \ref{replace} combined with Corollary \ref{cor-generalhern}.

\end{proof}

For $\pi\in \mathfrak{R}(G_n)$, we write $\pi^\vee\in \mathfrak{R}(G_n)$ for the contragredient (smooth dual) representation. Recall, that this involution of the category was composed with the Gelfand-Kazhdan involution of Section \ref{sect-gk} to obtain Proposition \ref{GK}.

For a segment $\Delta=[a,b]_\rho\in \seg$, we write $\Delta^\vee = [-b,-a]_{\rho^\vee}$. This is an involution which naturally extends to $\multi$. When $\pi = L(\gotM)=Z(\gotN)\in \irr$, we have $\pi^\vee = L(\gotM^\vee)= Z(\gotN^\vee)$ (in particular, the operation of smooth dual commutes with the Zelevinski involution, that is, sending $Z(\gotM)$ to $L(\gotM)$).

Let $\rho\in \irr^u\,\cap\, \mathcal{C}$ be given. We would like to define a similar involution on $\multi_\rho$. Given $\Delta= [a,b]_{\rho\nu^s}\in \seg_\rho$, we set
\[
\Delta^! = [-b,-a]_{\rho\nu^{-s}}\;,
\]
and extend it to an involution $\gotM\mapsto \gotM^!$, for all $\gotM\in \multi_\rho$.

When $\rho\cong\rho^\vee$, we clearly have $\gotM^! = \gotM^\vee$.

Now, given $\gotM\in \multi_\rho$, we can write a unique decomposition $\gotM = \gotM_s + \gotM_a$, with $\gotM_s,\gotM_a\in \multi_\rho$, so that $\gotM_s^!= \gotM_s$ and $\gotM_s$ is the maximal multisegment with that property.

Clearly, for a segment $\Delta\in \gotM_a$, we have $\Delta^!\not\in \gotM_a$.

\begin{lemma}\label{seg-lem}

Suppose that $\pi$ is a quasi-Arthur-type representation.

If $\overline{\pi}:=\nu^{-1/2}\pi^\vee$ is quasi-Arthur-type as well, given as
\[
\overline{\pi} = L\left( \gotM_\rho^{a_1,b_1,c_1} + \ldots + \gotM_\rho^{a_k,b_k,c_k}\right)\;,
\]
then we can write
\[
\pi = L\left( \sum_{i\in I^-}\gotM_{\rho^\vee}^{a_i,b_i-1} + \sum_{i\in I^+} \gotM_{\rho^\vee}^{a_i,b_i+1,0}\right)\;,
\]
for a certain disjoint partition $\{1,\ldots,k\}=I^- \cup I^+$, where $\gotM_\rho^{a,o}$ is viewed as an empty multisegment.

In particular,
\[
\pi  \cong L(\gotM + \gotM'\nu^{-1/2}) \;,
\]
for some proper Arthur-type representations $L(\gotM),\,L(\gotM')$.
\end{lemma}

\begin{proof}
We can assume that
\[
\pi = L\left( \gotM_{\rho^\vee}^{a'_1,b'_1,c'_1} + \ldots + \gotM_{\rho^\vee}^{a'_m,b'_m,c'_m}\right)\;,
\]
for some integers $a'_i,b'_i,c'_i$.
Let us write $\gotN = \left(\gotM_{\rho^\vee}^{a'_1,b'_1,c'_1}+ \ldots + \gotM_{\rho^\vee}^{a'_m,b'_m,c'_m}\right)\nu^{1/2}$. Clearly, we have
\[
\gotN_s = \gotM_{\rho^\vee}^{a'_1,b'_1-1} + \ldots + \gotM_{\rho^\vee}^{a'_m,b'_m-1}\;,
\]
and for every segment $\Delta\in \gotN_a$, we have
\begin{equation}\label{avg}
b_{\rho^\vee}(\Delta)+ e_{\rho^\vee}(\Delta) >0\;.
\end{equation}

From $\nu^{1/2}\pi =  \overline{\pi}^\vee$ we deduce that
\[
\gotN = (\gotM_{\rho}^{a_1,b_1,c_1})^\vee + \ldots + (\gotM_{\rho}^{a_k,b_k,c_k})^\vee =
\]
\[
= \gotM_{\rho^\vee}^{a_1, b_1-2} + \ldots + \gotM_{\rho^\vee}^{a_k,b_k-2} + \Delta^+_1 +\ldots + \Delta^+_k + \Delta^-_1 + \ldots +\Delta^-_k \;,
\]
where
\[
\Delta^+_i = \left\{\begin{array}{ll} \left[ \frac{b_i-a_i}2, \frac{b_i+a_i}2-1\right]_{\rho^\vee} & b_i>1 \\ 0 & b_i=1\end{array}\right. ,\quad  \Delta^-_i = \left[-\frac{b_i+a_i}2+1, -\frac{b_i+a_i}2+c_i\right]_{\rho^\vee}\;,
\]
and $\gotM^{a,0}_{\rho^\vee},\gotM^{a,-1}_{\rho^\vee}$ are understood as empty multisegments. Clearly,
\[
\mathfrak{t}:=\gotM_{\rho^\vee}^{a_1, b_1-2} + \ldots + \gotM_{\rho^\vee}^{a_k,b_k-2}  \subseteq \gotN_s\;.
\]

Note, that $b_{\rho^\vee}(\Delta^-_i) + e_{\rho^\vee}(\Delta^-_i)\leq 0$, for all $i$ with non-empty $\Delta^-_i$.

Suppose that $\Delta^-_i$ is such a non-empty segment. Then, by \eqref{avg} $\Delta^-_i\not\in \gotN_a$, which means $\Delta^-_i\in \gotN_s$. In case that $b_1 =1$ and $c_i=a_i$, we have $\Delta^-_i = \gotM_{\rho^\vee}^{a_i,1}= (\gotM_{\rho^\vee}^{a_i,1})^!$.

Otherwise, $\Delta^-_i\neq(\Delta^-_i)^!\in \gotN$. Since $\mathfrak{t}^!= \mathfrak{t}$ and $b_{\rho^\vee}((\Delta^-_i)^!)+  e_{\rho^\vee}((\Delta^-_i)^!)>0$, we must have $(\Delta^-_i)^! = \Delta^+_j$, for some $j$.

It follows that
\[
\gotM_{\rho^\vee}^{a_j,b_j}= \gotM_{\rho^\vee}^{a_j,b_j-2} + \Delta^-_i + \Delta^+_j\subseteq \gotN_s\;.
\]
We then can write
\[
\gotN_s = \sum_{j\in J} \gotM_\rho^{a_j, b_j-2} + \sum_{j\in K} \gotM_\rho^{a_j,b_j}\;,\quad \gotN_a = \sum_{j\in J}\Delta^+_j\;,
\]
for disjoint subsets $J,K\subset\{1,\ldots,k\}$, such that $\{\Delta^+_j\}_{j\in J}$ are non-empty.

Note, that $b_i=1$ for all $i\in \{1,\ldots,k\}\setminus(J\cup K)$.

The statement follows from the observations
\[
\nu^{-1/2}\gotM_{\rho^\vee}^{a_j,b_j} = \gotM_{\rho^\vee}^{a_j,b_j+1,0},\quad \nu^{-1/2}(\gotM_{\rho^\vee}^{a_j, b_j-2}+\Delta^+_j)= \gotM_{\rho^\vee}^{a_j,b_j-1}\;.
\]
\end{proof}

\begin{corollary}\label{cor-main}
Suppose that
\[
\pi_1 = L\left( \gotM^{a_1,b_1,c_1}_\rho + \ldots + \gotM^{a_k,b_k,c_k}_\rho\right)\;,
\]
\[
\pi_2 = L\left( \gotM^{a'_1,b'_1,c'_1}_{\rho^\vee} + \ldots + \gotM^{a'_l,b'_l,c'_l}_{\rho^\vee}\right)\;,
\]
are two quasi-Arthur-type representations, satisfying $\pi_2 = \nu^{-1/2}\pi^\vee_1$.

Then, there are disjoint partitions $\{1,\ldots,k\} = I_1\cup I_2\cup I_3$, $\{1,\ldots,l\} = J_1\cup J_2\cup J_3$ and bijections $u:I_1\to J_2$, $d:I_2\to J_1$, which satisfy
\[
(a'_{u(i)}, b'_{u(i)}) = (a_i,b_i+1)\; \forall i\in I_1, (a'_{d(i)},b'_{d(i)}) = (a_i,b_i-1)\; \forall i\in I_2,\;
\]
\[
b_i =1 \; \forall i\in I_3,\quad b'_j=1\;\forall j\in J_3\;.
\]

\end{corollary}

\begin{proof}
By Lemma \ref{seg-lem} we have an equality of multisegments of the form
\[
\sum_{j\in J_1}\gotM_\rho^{a'_j,b'_j+1,0}+ \sum_{j\in J_2} \gotM_\rho^{a'_{j},b'_{j}-1} = \gotM_{\rho}^{a_1,b_1,c_1} + \ldots + \gotM_{\rho}^{a_k,b_k,c_k}\;,
\]
for a disjoint partition $\{1,\ldots,l\} = J_1\cup J_2\cup J_3$, such that $b'_j=1$ for all $j\in J_3$, and $b'_j>1$ for all $j\in J_2$.

Let us further write a disjoint partition $J_2 = J'_2 \cup J''_2$, where $J'_2= \{j\in J_2\,:\, b'_j=2\}$.
Similarly, we write $\{1,\ldots,k\} = I'\cup I''$, where $I'= \{i\in I\,:\, b_i=1\}$.

For a multisegment $\gotM = \sum_{s\in K}[\alpha_s,\beta_s]_\rho \in \seg_\rho$, let us write a decomposition
\[
\gotM = \gotM_- + \gotM_0 + \gotM_+\;,
\]
\[
\gotM_- = \sum_{s\in K,\; \alpha_s +\beta_s <0} [\alpha_s,\beta_s]_\rho\;,\;\gotM_0 = \sum_{s\in K,\; \alpha_s +\beta_s =0} [\alpha_s,\beta_s]_\rho\;,\;\gotM_+ = \sum_{s\in K,\; \alpha_s +\beta_s >0} [\alpha_s,\beta_s]_\rho\;.
\]
We obtain the identity
\[
\left(\gotM_{\rho}^{a_1,b_1} + \ldots + \gotM_{\rho}^{a_k,b_k}\right)_- = \left(\gotM_{\rho}^{a_1,b_1,c_1} + \ldots + \gotM_{\rho}^{a_k,b_k,c_k}\right)_- =
\]
\[
\left(\sum_{j\in J_1}\gotM_\rho^{a'_j,b'_j+1,0}+ \sum_{j\in J_2} \gotM_\rho^{a'_{j},b'_{j}-1}\right)_- = \left(\sum_{j\in J_1}\gotM_\rho^{a'_j,b'_j+1}+ \sum_{j\in J_2} \gotM_\rho^{a'_{j},b'_{j}-1}\right)_-\;.
\]
A moment's reflection will show that such an identity forces a bijection $t:I''\to J_1\cup J''_2$, so that $(a_i, b_i)= (a'_{t(i)}, b'_{t(i)}+1)$, in case that $t(i)\in J_1$, and $(a_i, b_i)= (a'_{t(i)}, b'_{t(i)}-1)$, in case that $t(i)\in J''_2$. 

In particular, we see that
\[
\left(\sum_{j\in J_1}\gotM_\rho^{a'_j,b'_j+1,0}+ \sum_{j\in J''_2} \gotM_\rho^{a'_{j},b'_{j}-1}\right)_0 = \left(\sum_{j\in J_1}\gotM_\rho^{a'_j,b'_j+1}+ \sum_{j\in J''_2} \gotM_\rho^{a'_{j},b'_{j}-1}\right)_0 =
\]
\[
= \left( \sum_{i\in I''} \gotM_\rho^{a_i,b_i}\right)_0 =\left( \sum_{i\in I''} \gotM_\rho^{a_i,b_i,c_i}\right)_0 \;,
\]
which also implies
\[
\sum_{j\in J'_2} \gotM_\rho^{a'_{j},1} = \left(\sum_{j\in J'_2} \gotM_\rho^{a'_{j},1}\right)_0 = \left( \sum_{i\in I'} \gotM_\rho^{a_i,1,c_i}\right)_0\;.
\]
Since
\[
\left(\gotM^{a,1,c}_\rho\right)_0= \left\{ \begin{array}{cc} \gotM_\rho^{a,1} & c=a \\ 0 & c<a\end{array}\right.\;,
\]
we see that $\sum_{j\in J'_2} \gotM_\rho^{a'_{j},1} = \sum_{i\in \tilde{I}} \gotM_\rho^{a_i,1}$, for a subset $\tilde{I}\subset I'$. This implies a bijection $s: \tilde{I}\to J'_2$, such that $(a_i,b_i) = (a'_{s(i)}, b'_{s(i)}-1)$, for all $i\in \tilde{I}$.

The desired bijections $u$ and $d$ are easily constructed out of $t$ and $s$ after denoting $I_1 = \tilde{I}\cup t^{-1}(J''_2)$, $I_2 = t^{-1}(J_1)$ and $I_3 = I_1\setminus \tilde{I}$.


\end{proof}

\section{Main Theorems}\label{sect-main}

We would like to determine for which pairs of representations $\pi_1,\pi_2\in \irr^u$, such that $\pi_1\in \irr(G_n)$ and $\pi_2\in \irr(G_{n-1})$, the space
\[
\Hom (\pi_1|_{G_{n-1}},\pi_2)
\]
is non-zero. Here we consider $G_{n-1}$ as a subgroup of $G_n$ embedded in the corner, as described in the introduction section.

Now, suppose that $\pi_1,\pi_2$ are Arthur-type representations. It is easy to verify that they can be written uniquely (up to ordering) in the form
\[
\pi_1 = L\left( \gotM^{a_1,b_1}_{\rho_1} + \ldots + \gotM^{a_k,b_k}_{\rho_k}\right)\;,
\]
\[
\pi_2 = L\left( \gotM^{a'_1,b'_1}_{\rho'_1} + \ldots + \gotM^{a'_l,b'_l}_{\rho'_l}\right)\;.
\]
\begin{definition}
We say that a pair $(\pi_1,\pi_2)$ of Arthur-type representations is in \textit{GGP position}, if in terms of the decomposition above, there are disjoint partitions
\[
\{1,\ldots,k\} = I_1\cup I_2\cup I_3, \quad\{1,\ldots,l\} = J_1\cup J_2\cup J_3\;,
\]
and bijections $u:I_1\to J_2$, $d:I_2\to J_1$, which satisfy
\[
(a'_{u(i)}, b'_{u(i)}) = (a_i,b_i+1),\,\rho'_{u(i)}\cong \rho_i\; \forall i\in I_1,
\]
\[
(a'_{d(i)},b'_{d(i)}) = (a_i,b_i-1),\,\rho'_{d(i)}\cong \rho_i\; \forall i\in I_2,\;
\]
\[
b_i =1 \; \forall i\in I_3,\quad b'_j=1\;\forall j\in J_3\;.
\]
\end{definition}

Our terminology \textit{GGP position} stands for the authors' names Gan-Gross-Prasad, who have coined this condition in their Conjecture \ref{main-conj}.

We are now going to prove two main intermediate steps that together will imply that a non-zero $\Hom$ space between two Arthur-type representations $\pi_1,\pi_2$ of $G_n,G_{n-1}$ must imply that the pair $(\pi_1,\pi_2)$ is in GGP position.

The first step (Section \ref{sect-filtr}, Proposition \ref{filt-frob}) will transfer the non-vanishing of the restricted $\Hom$ space into the non-vanishing of another $\Hom$ space between certain derivatives of $\pi_1,\pi_2$. This will have the immense advantage of translating the problem into one about the category of finite-length representations of a smaller group. It will be possible due the useful Bernstein-Zelevinski filtration occurring when restricting to the mirabolic group situated amidst $G_{n-1}$ and $G_n$.

The second step is proved in the following theorem, that states (using the tools developed in Section \ref{sect-work}) that the non-zero derivative $\Hom$ space implies a combinatorial condition that can be read as the GGP position.

Put together, these steps will be summed up in Section \ref{sect-branch} to state our main results.

\begin{theorem}\label{thm-deri}
Suppose that $\pi_1,\pi_2\in \irr^u$ are two Arthur-type representations, which satisfy
\[
\Hom (\nu^{1/2}\pi^{(i)}_1, \; ^{(j)}\pi_2)\neq \{0\}\;,
\]
for some $i,j$.

Then, $(\pi_1,\pi_2)$ is in GGP position.

\end{theorem}

\begin{proof}
Let us write $\pi_1\cong \pi_{\rho'_1}^{a'_1,b'_1}\times \cdots \times \pi_{\rho'_k}^{a'_k,b'_k}$ and $\pi_2\cong \pi_{\rho_1}^{a_1,b_1}\times \cdots \times \pi_{\rho_l}^{a_l,b_l}$.

Since the above products are independent of the order in which the factors are taken, we can assume that $a_i+b_i\geq a_j+b_j$ and $a'_i+b'_i\geq a'_j+b'_j$ for all $i<j$ and that $a_i\geq a_j$ (respectively, $a'_i\geq a'_j$), in case that $a_i+b_i= a_j+b_j$ (respectively, $a'_i+b'_i= a'_j+b'_j$).

Recall that $(\pi^{a,b}_\rho)^\vee \cong \pi^{a,b}_{\rho^\vee}$. Then, by Propositions \ref{leibn}, \ref{deriv-speh} and the assumption on the non-vanishing homomorphism space, we deduce that
\[
\Hom \left( \nu^{1/2}\left( \pi_{\rho'_1}^{a'_1,b'_1,c'_1}\times \cdots \times \pi_{\rho'_l}^{a'_l,b'_l,c'_l}\right), \,  \left(\pi_{\rho^\vee_1}^{a_1,b_1,c_1}\times \cdots \times \pi_{\rho^\vee_k}^{a_k,b_k,c_k}\right)^\vee\right)\neq \{0\}\;,
\]
for some integers $c_1,\ldots,c_k,c'_1,\ldots,c'_l$.

By rearranging the factors in the product if necessary and using Corollary \ref{disjsupp-cor}, we can decompose the above homomorphism space as
\[
\bigotimes_{i=1}^t \Hom \left( \nu^{1/2}\left( \pi_{\rho^i}^{a'_{r'_i(1)},b'_{r'_i(1)},c'_{r'_i(1)}}\times \cdots \times \pi_{\rho^i}^{a'_{r'_i(s'_i)},b'_{r'_i(s'_i)},c'_{r'_i(s'_i)}} \right),\right.
\]
\[
\left. \left(\pi_{\left(\rho^i\right)^\vee}^{a_{r_i(1)},b_{r_i(1)},c_{r_i(1)}}\times \cdots \times \pi_{\left(\rho^i\right)^\vee}^{a_{r_i(s_i)},b_{r_i(s_i)},c_{r_i(s_i)}}\right)^\vee\right)\;,
\]
where $\{\rho^1,\ldots,\rho^t\} = \{\rho_i\}_{i=1}^k\cup \{\rho'_i\}_{i=1}^l$, such that the $\rho^i$'s are pairwise non-isomorphic unitary supercuspisdal representations, $r_i:\{1,\ldots,s_i\}\to \{1,\ldots,k\}$ are ascending injections such that $\{1,\ldots,k\} = \dot{\cup}_i \im r_i$\footnote{Here $\dot{\cup}$ stands for a disjoint union.} and $\rho_{r(i)}\cong \rho^i$ , and $r'_i:\{1,\ldots,s'_i\}\to \{1,\ldots,l\}$ are ascending injections such that $\{1,\ldots,l\} = \dot{\cup}_i \im r'_i$ and $\rho'_{r'(i)}\cong \rho^i$.

In particular, the above decomposition shows that it is enough to consider the case where all $\rho_i$'s and $\rho'_j$'s are isomorphic. In other words, we are free to assume for the rest of the proof that $\pi_1,\pi_2$ are of proper Arthur-type and that $\rho_i\cong \rho'_j\cong \rho$ for all $i,j$ and one fixed $\rho \in \mathcal{C}\,\cap \irr^u$.

Now, note that the given quasi-Speh representations satisfy $\pi_{\rho^\vee}^{a_k,b_k,c_k}\preceq \ldots\preceq \pi_{\rho^\vee}^{a_1,b_1,c_1}$ and $\pi_\rho^{a'_l,b'_l,c'_l}\preceq \ldots\preceq \pi^{a'_1,b'_1,c'_1}_\rho$.

Hence, by Proposition \ref{main-prop}, we have the quasi-Arthur type representation
\[
\sigma_2:=L(\gotM_{\rho^\vee}^{a_1,b_1,c_1} + \ldots + \gotM_{\rho^\vee}^{a_k,b_k,c_k})
\]
as the unique irreducible quotient of $\pi_{\rho^\vee}^{a_1,b_1,c_1} \times \cdots\times \pi_{\rho^\vee}^{a_k,b_k,c_k}$, while
\[
\sigma_1:=L(\gotM_\rho^{a'_1,b'_1,c'_1} + \ldots + \gotM_\rho^{a'_l,b'_l,c'_l})
\]
as the unique irreducible quotient of $\pi_\rho^{a'_1,b'_1,c'_1} \times \cdots\times \pi_\rho^{a'_l,b'_l,c'_l}$.

By applying Lemma \ref{mult-techlem}, we can deduce from the non-vanishing $\Hom$ space that $\sigma_1\nu^{1/2}\cong \sigma^\vee_2$. Finally, the statement follows from Corollary \ref{cor-main}.

\end{proof}

\subsection{Bernstein-Zelevinski filtration}\label{sect-filtr}

In order to study the morphism space above, we will make use of the analysis in \cite{BZ1} of restrictions of representations in $\irr$. Let us sketch the main ingredients used in that reference. The reader can also refer to \cite{chan-savin,prasad-ext} for very similar discussions.

Recall the $p$-adic \textit{mirabolic} groups $\{P_n\}$ situated within the inclusions $G_{n-1}<P_n< G_n$. In matrix form, $P_n$ is defined to be the subgroup of $G_n$ consisting of matrices whose bottom row is given by $(0\,\ldots\,0\,1)$.

Bernstein-Zelevinski defined families of exact functors
\[
\Phi^-: \mathfrak{S}(P_{n}) \to  \mathfrak{S}(P_{n-1})\qquad \Phi^+, \hat{\Phi}^+: \mathfrak{S}(P_{n}) \to  \mathfrak{S}(P_{n+1})
\]
\[
\Psi^-: \mathfrak{S}(P_{n}) \to  \mathfrak{S}(G_{n-1})\qquad \Psi^+: \mathfrak{S}(G_{n}) \to  \mathfrak{S}(P_{n+1})\;,
\]
together with a set of identities between them. We will not use the definition of these functors, but let recall they are all defined through suitable induction (in case of $+$ functors) or restriction of coinvariants (in case of $-$ functors). The functor $\Phi^+$ is the sub-functor of $\hat{\Phi}^+$ that takes the compactly supported sections of the induced representation.

In particular, $\Phi^-$ is both left adjoint to $\hat{\Phi}^+$ and right adjoint to $\Phi^+$, while $\Psi^-$ is left adjoint to $\Psi^+$.

In terms of these functors, for every $0<i\leq n$, the $i$-th derivative of a representation $\pi\in\mathfrak{R}(G_n)$ is given as
\[
\pi^{(i)}= \Psi^-(\Phi^-)^{i-1}(\pi|_{P_n})\in \mathfrak{R}(G_{n-i})\;,
\]
where $(\Phi^-)^{i-1}$ denotes the $i-1$-th consecutive composition of $\Phi^-$.

\begin{proposition}\label{BZfilt}
  For any representation $\tau\in \mathfrak{R}(P_n)$ on a vector space $V$, there is a filtration of $P_n$-representations
  \[
  \{0\}=V_{n}\subset V_{n-1}\subset \ldots \subset V_0 = V\;,
  \]
  such that
  \[
  V_i/V_{i+1} \cong  (\Phi^+)^i\Psi^+\Psi^-(\Phi^-)^{i}(\tau)\;,
  \]
  as a $P_n$-representation, for all $0\leq i< n$.
\end{proposition}

Following the terminology of \cite{chan-savin}, let us define the \textit{left }$i$-th \textit{derivative} of $\pi\in \mathfrak{R}(G_n)$ by setting
\[
^{(i)}\pi := \left((\pi^\vee)^{(i)}\right)^\vee\;.
\]

\begin{proposition}\label{filt-frob}
  Let $\pi_1\in \mathfrak{R}(G_n)$ and $\pi_2\in \mathfrak{R}(G_{n-1})$ be given. Let $\{V_i\}$ be the filtration of $\pi_1|_{P_n}$ as described in Proposition \ref{BZfilt}. Then, we have a natural identification of homomorphism spaces
  \[
  \Hom_{G_{n-1}}(V_i/V_{i-1}, \pi_2)\cong \Hom_{G_{n-i-1}} (\nu^{1/2}\pi_1^{(i+1)}, ^{(i)}\pi_2)\;,
  \]
  for all $0\leq i< n$.

  In particular, a non-zero element in $\Hom_{G_{n-1}} (\pi_1|_{G_{n-1}},\pi_2)$ gives rise to a non-zero element in $\Hom_{G_{n-i-1}} (\nu^{1/2}\pi_1^{(i+1)}, ^{(i)}\pi_2)$, for some $i$.

\end{proposition}

\begin{proof}
Let $I$ (respectively, $i$) denote the induction (respectively, compact induction) functor from $\mathfrak{R}(G_{n-1})$ to $\mathfrak{R}(P_n)$. Recall that by standard Mackey theory, the restriction functor $\mathfrak{R}(P_n)\to \mathfrak{R}(G_{n-1})$ is left adjoint to $I$ (see, e.g. \cite[1.9]{BZ1}).

Thus, by Proposition \ref{BZfilt},
\[
\Hom_{G_{n-1}}(V_i/V_{i-1}, \pi_2)\cong \Hom_{G_{n-1}}\left(\left.\left((\Phi^+)^i\Psi^+\left(\pi_1^{(i+1)}\right)\right)\right|_{G_{n-1}}, \pi_2\right)\cong
\]
\[
\cong \Hom_{P_{n}}\left((\Phi^+)^i\Psi^+\left(\pi_1^{(i+1)}\right), I(\pi_2)\right)\cong \Hom_{P_{n}}\left( I(\pi_2)^\vee, \left((\Phi^+)^i\Psi^+\left(\pi_1^{(i+1)}\right)\right)^\vee  \right)\;.
\]
Now, by \cite[3.4]{BZ1}, we know that
\[
\left((\Phi^+)^i\Psi^+\left(\pi_1^{(i+1)}\right)\right)^\vee \cong \nu^{-1} (\hat{\Phi}^+)^i\Psi^+\left(\left(\pi_1^{(i+1)}\right)^\vee\right)\cong \nu^{-1/2} (\hat{\Phi}^+)^i\Psi^+\left(\nu^{-1/2}\left(\pi_1^{(i+1)}\right)^\vee\right)
\]
holds, while similarly $I(\pi_2)^\vee\cong \nu^{-1}i(\pi_2^\vee)\cong \nu^{-1/2}i(\nu^{-1/2}\pi_2^\vee)$.

From the above mentioned adjunctions of functors, we now see that
\[
\Hom_{G_{n-1}}(V_i/V_{i-1}, \pi_2)\cong \Hom_{P_{n}}\left( i(\nu^{-1/2}\pi_2^\vee), (\hat{\Phi}^+)^i\Psi^+\left(\nu^{-1/2}\left(\pi_1^{(i+1)}\right)^\vee \right) \right) \cong
\]
\[
\cong\Hom_{G_{n-i-1}} \left(  \Psi^-(\Phi^-)^i\left( i(\nu^{-1/2}\pi_2^\vee)\right) ,  \nu^{-1/2} \left(\pi_1^{(i+1)}\right)^\vee\right) \;.
\]
Yet, it follows from \cite[4.13(c)]{BZ1} that $\Phi^-\left( i(\nu^{-1/2}\pi_2^\vee)\right)\cong\pi_2^\vee|_{P_{n-1}}$. Hence,
\[
\Hom_{G_{n-1}}(V_i/V_{i-1}, \pi_2)\cong  \Hom_{G_{n-i-1}} \left( \left(\pi_2^\vee\right)^{(i)} ,   \nu^{-1/2}\left(\pi_1^{(i+1)}\right)^\vee\right) \cong
\]
\[
\cong \Hom_{G_{n-i-1}} (\nu^{1/2}\pi_1^{(i+1)},\; ^{(i)}\pi_2)\;.
\]

\end{proof}

\begin{corollary}\label{filt-frob-cor}
   Let $\pi_1\in \mathfrak{R}(G_n)$ and $\pi_2\in \mathfrak{R}(G_{n-1})$ be given. Let $\{V_i\}$ be the filtration of $\pi_1|_{P_n}$ as described in Proposition \ref{BZfilt}. Then, we have a natural identification of homology spaces
  \[
  \Ext^1_{G_{n-1}}(V_i/V_{i-1}, \pi_2)\cong \Ext^1_{G_{n-i-1}} (\nu^{1/2}\pi_1^{(i+1)}, ^{(i)}\pi_2)\;,
  \]
  for all $0\leq i< n$.
\end{corollary}

\begin{proof}
The proof of Proposition \ref{filt-frob} will work for this case as well, after recalling that adjunctions of exact functors give natural identifications of $\Ext$-spaces, as well as $\Hom$-spaces. See e.g. \cite[Proposition 2.2]{prasad-ext}.

\end{proof}

\subsection{Branching laws}\label{sect-branch}
We will first treat the Arthur-type case which was considered in Conjecture \ref{main-conj}.
\begin{theorem}\label{thm-main}
Let $\pi_1,\pi_2\in \irr^u$ be two Arthur-type representations, such that $\pi_1\in \irr(G_n)$ and $\pi_2\in \irr(G_{n-1})$.

If
\[
\Hom_{G_{n-1}} (\pi_1|_{G_{n-1}},\pi_2)\neq \{0\}
\]
holds, then $(\pi_1,\pi_2)$ must be in GGP position.
\end{theorem}
\begin{proof}
The last part of Proposition \ref{filt-frob} implies that under the stated condition \newline $\Hom_{G_{n-i-1}} (\nu^{1/2}\pi_1^{(i+1)}, ^{(i)}\pi_2)$ must be non-zero, for a certain $i$. Hence, the result follows from Theorem \ref{thm-deri}.
\end{proof}

The above theorem can also be extended into a branching law governing the irreducible unitarizable quotients of a restriction of any unitarizable irreducible representation.

\begin{theorem}\label{thm-main-ext}
Let $\pi,\sigma\in \irr^u$ be two representations, such that $\pi\in \irr(G_n)$ and $\sigma\in \irr(G_{n-1})$.

Let
 \[
\pi \cong \pi_0 \times \pi_1(\alpha_1)\times\cdots\times \pi_k(\alpha_k)\;,
 \]
 \[
\sigma \cong \sigma_0 \times \sigma_1(\beta_1)\times\cdots\times \sigma_l(\beta_l)\;,
 \]
be their decomposition as described in Section \ref{sect-tadic}, that is, $\pi_i, \sigma_i\in \irr^u$ are Arthur-type representations and $\{\alpha_i\},\,\{\beta_i\}$ are sets of real numbers.

If
\[
\Hom_{G_{n-1}} (\pi|_{G_{n-1}},\sigma)\neq \{0\}
\]
holds, then
\begin{enumerate}
  \item\label{item1} The pair $(\pi_0,\sigma_0)$ is in GGP position.
  \item\label{item2} For every $i,j$ for which $\alpha_i = \beta_j$ holds, the pair $(\pi_i, \sigma_j)$ is in GGP position.
  \item\label{item3} If $\alpha_i\not\in \{\beta_1,\ldots,\beta_l\}$ for some $i$, then $\pi_i$ is generic (equivalently, $\pi_i$ is tempered, or equivalently, $\pi_i(\alpha_i)$ is generic), i.e.
  \[
  \pi_i = L(\gotM^{a_1,1}_{\rho_1} + \ldots + \gotM^{a_t,1}_{\rho_t})\;.
  \]
  \item\label{item4} Similarly, if $\beta_j\not\in \{\alpha_1,\ldots,\alpha_k\}$ for some $j$, then $\sigma_j$ is generic.
\end{enumerate}

\end{theorem}

\begin{proof}
By Proposition \ref{filt-frob} we know that $\Hom (\nu^{1/2}\pi^{(i+1)}, ^{(i)}\sigma)$ is non-zero, for some $i$.

By Proposition \ref{leibn} that means
\[
H:= \Hom \left(\nu^{1/2}\left(\pi_0^{(c_0)}\times\pi_1(\alpha_1)^{(c_1)}\times\cdots\times \pi_k(\alpha_k)^{(c_k)}\right) ,\right.
\]
\[
\left. ^{(d_0)}\sigma_0 \times  \,^{(d_1)}\left(\sigma_1(\beta_1)\right)\times\cdots\times \,^{(d_l)}\left(\sigma_l(\beta_l)\right)\right)\neq \{0\}
\]
holds, for some $c_0,\ldots,c_k,d_0,\ldots,d_l$.

It is easy to see that given Arthur-type representations $\tau, \tau',\tau''$ and numbers $0<\gamma < \gamma'<1/2$, the representations $\tau, \tau'(\gamma'), \tau''(\gamma'')$ have pairwise disjoint supecuspidal supports. The same clearly holds for their derivatives.

Hence, for $\alpha_i\not\in \{\beta_1,\ldots,\beta_l\}$, we must have $\pi_i(\alpha_i)^{(c_i)}\in \irr(G_0)$. Claim \eqref{item3} then follows, with claim \eqref{item4} following similarly.

Also, arguing as in the proof of Theorem \ref{thm-deri}, using Corollary \ref{disjsupp-cor}, we see that
\[
H = \Hom \left( \nu^{1/2} \pi_0^{(c_0)}, \,^{(d_0)}\sigma_0\right)\otimes\left( \bigotimes_{(i,j)\,:\, \alpha_i=\beta_j} \Hom \left( \nu^{1/2}\pi_i(\alpha_i)^{(c_i)}, \,^{(d_j)}\left(\sigma_j(\beta_j)\right)\right)\right)
\]

Claim \eqref{item1} now follows from Theorem \ref{thm-deri}.

In order to establish claim \eqref{item2}, we need to note that for $\alpha = \alpha_i = \beta_j$, by Proposition \ref{leibn} and Corollary \ref{disjsupp-cor}, we have
\[
\Hom \left( \nu^{1/2}\pi_i(\alpha)^{(c_i)}, \,^{(d_j)}\left(\sigma_j(\alpha)\right)\right)\neq \{0\}
\implies \left\{\begin{array}{ll} \Hom \left( \nu^{1/2+\alpha}\pi_i^{(c)}, \nu^{\alpha}\,\left(^{(d)}\sigma_j\right)\right)\neq\{0\} \\ \Hom \left( \nu^{1/2-\alpha}\pi_i^{(c')}, \nu^{-\alpha}\,\left(^{(d')}\sigma_j\right)\right)\neq\{0\} \end{array}\right.\;,
\]
for some $c+c'=c_i$ and $d+d' =d_j$.
The claim then follows again from Theorem \ref{thm-deri}.

\end{proof}

\subsection{Converse direction}

Let us first state a certain corollary from a discussion in \cite{LM2}.
\begin{lemma}\label{LMlem}
  Let $\sigma_1,\sigma_2\in \irr^u$ be two Speh representations. Then, $\nu^{1/2}\sigma_1\times \sigma_2$ is irreducible. In particular, $\nu^{1/2}\sigma_1\times \sigma_2 \cong \sigma_2\times \nu^{1/2}\sigma_1$.
\end{lemma}
\begin{proof}
  Let us write $\sigma_1 = L(\gotM^{a_1,b_1}_{\rho_1})$ and $\sigma_2= L(\gotM^{a_2,b_2}_{\rho_2})$. If $\rho_1\not\cong \rho_2$ holds, the statement is clear. Otherwise, let us apply the irreducibility criterion from \cite[Corollary 5.10]{LM2}.

  For the purposes of this proof, let us view $\supp(\nu^{1/2})\sigma_1$ and $\supp(\sigma_2)$ as elements of $\seg$ in a natural way. Namely,
\[
\supp(\nu^{1/2}\sigma_1) = \left[-\frac{a_1+b_1-3}2, \frac{a_1+b_1-1}2\right]_{\rho_1}\;,\quad\supp(\sigma_2) = \left[-\frac{a_2+b_2-2}2, \frac{a_2+b_2-2}2\right]_{\rho_2}\;.
\]
By \cite[Corollary 5.10]{LM2}, it is enough to show that as segments, $\supp(\nu^{1/2}\sigma_1)$ and $\supp(\sigma_2)$ do not precede each other. If $a_1+b_1$ and $a_2+b_2$ have distinct parities, it follows immediately. Otherwise, it is easy to see that the inequalities implied by demanding either of the segments to precede the other do not have a solution.

\end{proof}

\begin{proposition}\label{indeed-quo}
Let $(\pi_1,\pi_2)$ be a pair of Arthur-type representations in GGP position, such that $\pi_1\in \irr(G_n)$ and $\pi_2\in \irr(G_{n-1})$.

Let $\{V_i\}$ be the filtration of $\pi_1|_{P_n}$ as described in Proposition \ref{BZfilt}.

Then,
\[
 \Hom_{G_{n-1}}(V_i/V_{i-1}, \pi_2)\neq \{0\}\;,
\]
for some $i$.
\end{proposition}
\begin{proof}
  From the assumption on the GGP position, we are able to write the Arthur-type representations as products of Speh representations of a certain form. Namely,
  \[
  \pi_1 = \pi^{a_1,b_1}_{\rho_1}\times \cdots \times  \pi^{a_t,b_t}_{\rho_t}\times \pi^{d_1,1}_{\varrho_1}\times\cdots\times  \pi^{d_s,1}_{\varrho_s}\times \pi^{a'_1,b'_1+1}_{\rho'_1}\times \cdots \times  \pi^{a'_{t'},b'_{t'}+1}_{\rho'_{t'}}\;,
  \]
  \[
  \pi_2 = \pi^{a'_1,b'_1}_{\rho'_1}\times \cdots \times  \pi^{a'_{t'},b'_{t'}}_{\rho'_{t'}}\times \pi^{d'_1,1}_{\varrho'_1}\times\cdots\times  \pi^{d'_{s'},1}_{\varrho'_{s'}}\times \pi^{a_1,b_1+1}_{\rho_1}\times \cdots \times  \pi^{a_{t},b_{t}+1}_{\rho_{t}}\;,
  \]
for some $\{\rho_i\}, \{\rho'_i\}, \{\varrho_i\}, \{\varrho'_i\}\subset \mathcal{C}\,\cap\,\irr^u$ and positive integers $\{a_i\},\{b_i\},\{a'_i\}, \{b'_i\},\{d_i\},\{d'_i\}$.

Now, let us also define the representations
\[
\Pi_1 =  \pi^{a_1,b_1}_{\rho_1}\times \cdots \times  \pi^{a_t,b_t}_{\rho_t} \times \nu^{-1/2} \pi^{a'_1,b'_1}_{\rho'_1}\times \cdots \times  \nu^{-1/2}\pi^{a'_{t'},b'_{t'}}_{\rho'_{t'}} \;,
\]
\[
\Pi_2 = \pi^{a'_1,b'_1}_{\left(\rho'_1\right)^\vee}\times \cdots \times  \pi^{a'_{t'},b'_{t'}}_{\left(\rho'_{t'}\right)^\vee} \times  \nu^{-1/2}\pi^{a_1,b_1}_{\rho^\vee_1}\times \cdots \times  \nu^{-1/2}\pi^{a_t,b_t}_{\rho^\vee_t}\;.
\]
By Propositions \ref{leibn} and \ref{deriv-speh}, $\Pi_1$ appears as a quotient of $\pi_1^{(i)}$, while $\Pi_2$ appears as a quotient of $\left(\pi^\vee_2\right)^{(j)}$, for a certain choice of $i,j$.

Note, that by successive application of Lemma \ref{LMlem}, we see that
\[
\Pi_2 \cong  \nu^{-1/2}\pi^{a_1,b_1}_{\rho^\vee_1}\times \cdots \times  \nu^{-1/2}\pi^{a_t,b_t}_{\rho^\vee_t} \times \pi^{a'_1,b'_1}_{\left(\rho'_1\right)^\vee}\times \cdots \times  \pi^{a'_{t'},b'_{t'}}_{\left(\rho'_{t'}\right)^\vee}\;,
\]
which, in particular, implies the isomorphism $\nu^{1/2}\Pi_1\cong \Pi^\vee_2$. Hence, $\Hom (\nu^{1/2}\pi^{(i)}_1, \,^{(j)}\pi_2)\neq\{0\}$.

Note also, that since $\Pi_1,\Pi_2$ turn out to be representations of the same group, we can deduce $j=i-1$ from the assumption that $\pi_1\in \irr(G_n)$ and $\pi_2\in \irr(G_{n-1})$. Finally, the statement follows from Proposition \ref{filt-frob}.

\end{proof}

Let $\pi_1\in \irr(G_n)$ and $\pi_2\in \irr(G_{n-1})$ be given. Let $\{V_i\}$ be the filtration of $\pi_1|_{P_n}$ as described in Proposition \ref{BZfilt}.

We would like to make the straightforward observation, that in order to show the non-vanishing of  $\Hom_{G_{n-1}} (\pi_1|_{G_{n-1}},\pi_2)$, it is enough to find a number $0\leq i <n$, for which the conditions
\begin{equation}\label{eqfilt}
  \left\{\begin{array}{ll} \Hom_{G_{n-1}}(V_i/V_{i-1}, \pi_2)\neq \{0\} \\ \Ext^1_{G_{n-1}}(V_j/V_{j-1}, \pi_2)= \{0\}\,,\; & \forall\, 0\leq j <i  \end{array} \right.\;
\end{equation}
hold.

By Proposition \ref{filt-frob} and Corollary \ref{filt-frob-cor}, the set of conditions \eqref{eqfilt} is equivalent to

\begin{equation}\label{eqfilt2}
    \left\{\begin{array}{ll} \Hom_{G_{n-1}}(\nu^{1/2}\pi_1^{(i+1)}, ^{(i)}\pi_2)\neq \{0\} \\ \Ext^1_{G_{n-1}}(\nu^{1/2}\pi_1^{(j+1)}, ^{(j)}\pi_2)= \{0\}\,,\; & \forall\, 0\leq j <i  \end{array} \right.\;.
\end{equation}

Recall that an Arthur-type representation
\[
\pi = L\left( \gotM^{a_1,b_1}_{\rho_1} + \ldots + \gotM^{a_k,b_k}_{\rho_k}\right)
\]
is generic, when $b_i=1$, for all $i=1,\ldots,k$.


\begin{theorem}\label{thm-conv}
  Let $(\pi_1,\pi_2)$ be a pair of Arthur-type representations in GGP position, such that $\pi_1\in \irr(G_n)$ and $\pi_2\in \irr(G_{n-1})$ and at least one of $\pi_1,\pi_2$ is generic.
Then,
\[
\Hom_{G_{n-1}} (\pi_1|_{G_{n-1}},\pi_2)\neq \{0\}\;.
\]
\end{theorem}

\begin{proof}
Let us assume for simplicity that $\pi_1$ is generic.
The case of generic $\pi_2$ is proved by same arguments, while exchanging the roles of $\pi_1,\pi_2$.

From the assumptions we are able to write (in similarity with the proof of Proposition \ref{indeed-quo})
 \[
  \pi_1 = \pi^{a_1,1}_{\rho_1}\times \cdots \times  \pi^{a_t,1}_{\rho_t}\times \pi^{d_1,1}_{\varrho_1}\times\cdots\times  \pi^{d_s,1}_{\varrho_s}\;,
  \]
  \[
 \pi_2=\pi^{d'_1,1}_{\varrho'_1}\times\cdots\times  \pi^{d'_{s'},1}_{\varrho'_{s'}}\times \pi^{a_1,2}_{\rho_1}\times \cdots \times  \pi^{a_{t},2}_{\rho_{t}}\;,
  \]
for some $\{\rho_i\}, \{\varrho_i\}, \{\varrho'_i\}\subset \mathcal{C}\,\cap\,\irr^u$ and positive integers $\{a_i\},\{d_i\},\{d'_i\}$.

From Proposition \ref{indeed-quo} we know that $\Hom_{G_{n-1}}(V_{i_0}/V_{i_0-1}, \pi_2)\neq \{0\}$ holds, for some $i_0$. Hence, by the discussion above, it is enough to show that $\Ext^1_{G_{n-1}}(\nu^{1/2}\pi_1^{(i+1)}, ^{(i)}\pi_2)$ vanishes, for all $0\leq i <i_0$.

Let us assume the contrary, that is, $\Ext^1_{G_{n-1}}(\nu^{1/2}\pi_1^{(i_1+1)}, ^{(i_1)}\pi_2)\neq \{0\}$ for a given $0\leq i_1 <i_0$. By Propositions \ref{leibn} and \ref{deriv-speh}, that means we have
\[
\Ext^1_{G_{n-1}}(\nu^{1/2}\sigma_1, \sigma^\vee_2)\neq\{0\}\;,
\]
for some representations
 \[
  \sigma_1 = \pi^{a_1,1,c_1}_{\rho_1}\times \cdots \times  \pi^{a_t,1,c_t}_{\rho_t}\times \pi^{d_1,1,e_1}_{\varrho_1}\times\cdots\times  \pi^{d_s,1,e_s}_{\varrho_s}\;,
  \]
  \[
  \sigma_2 =  \pi^{d'_1,1,e'_1}_{(\varrho'_1)^\vee}\times\cdots\times  \pi^{d'_{s'},1,e'_{s'}}_{(\varrho'_{s'})^\vee}\times \pi^{a_1,2,f_1}_{\rho^\vee_1}\times \cdots \times  \pi^{a_{t},2,f_t}_{\rho^\vee_{t}}\;.
  \]
Since non-trivial extensions can occur only between representations with the same central character, we see that $\Re(\nu^{1/2}\sigma_1) = \Re(\sigma_2^\vee) = -\Re(\sigma_2)$. As a consequence, the equality
\[
\Re(\nu^{1/2}\sigma_2) = -\Re(\sigma_1) = -\sum_{i=1}^t \Re(\pi^{a_i,1,c_i}_{\rho_i}) - \sum_{i=1}^s \Re(\pi^{d_i,1,e_i}_{\varrho_i})
\]
holds. Yet,
\[
\Re(\pi^{a,1,c}_{\rho}) = \Re\left( \rho\nu^{a/2-c+1/2}\times \rho\nu^{a/2-c+3/2} \times\cdots\times \rho\nu^{a/2-1/2}\right)= k_\rho\sum_{i=1}^c \left(\frac{a+1}2-i\right)\geq 0\;,
\]
for all $\rho\in \mathcal{C}(G_{k_\rho})\,\cap\,\irr^u$ and integers $0\leq c\leq a$. Hence, $\Re(\nu^{1/2}\sigma_2)\leq 0$.

On the other hand, for all $1\leq i\leq t$, we have
\[
\Re\left( \nu^{1/2}\pi^{a_{i},2,f_i}_{\rho^\vee_i}\right) = \Re\left(\rho^\vee_i\nu^{-\frac{a-1}{2}} \times \rho^\vee_i\nu^{-\frac{a-1}{2}+1}\times	 \cdots\times \rho^\vee_i\nu^{\frac{a-1}{2}}\right) + \Re\left(\nu^1\pi^{a_{i},1,f_i}_{\rho^\vee_i}\right)=
\]
\[
=\Re\left(\nu^1\pi^{a_{i},1,f_i}_{\rho^\vee_i}\right)\geq\Re\left(\pi^{a_{i},1,f_i}_{\rho^\vee_i}\right)\geq0\;.
\]
Moreover, it is clear that the above inequality is strict, unless $f_i=0$.
Thus,
\[
\Re(\nu^{1/2}\sigma_2) = \sum_{i=1}^{s'} \Re\left( \nu^{1/2}\pi^{d'_i,1,e'_i}_{(\varrho'_i)^\vee}\right) + \sum_{i=1}^{t} \Re\left( \nu^{1/2}\pi^{a_{i},2,f_i}_{\rho^\vee_{i}}\right) = 0\;,
\]
with $f_i=0$ and $e'_i=0$ for all $i$.

Now, this implies that $\sigma_2$ appears only in the highest non-zero derivative of $\pi_2^\vee$. This is a contradiction to the facts that $i_1<i_0$ and that $^{(i_0)}\pi_2$ is non-zero.

\end{proof}

\section{Quantum affine algebras}\label{sect-qa}
We would like to prove Theorem \ref{transl} by applying results from the representation theory of quantum affine algebras.

\subsection{Setting}

Let us recall parts of the theory of quantum affine algebras and their finite-dimensional representations. We refer the reader to \cite{cha-pres-qa}, \cite[Chapter 12]{cha-pres-book} ,\cite{fren-resh} for comprehensive study and discussions of these objects.

We are interested in the Hopf $\mathbb{C}$-algebra $\mathcal{A}_{N,q}=U_q\left(\hat{\frak{sl}}_N\right)$, which is defined for a fixed parameter $q\in \mathbb{C}^\times$. Its definition involves generators and relations which depend on the Cartan matrix of the affine Lie algbera $\hat{\frak{sl}}_N$. We will refrain from stating the full definition, which can be easily found in the sources mentioned above, since our applications will not require it. Our assumption for the rest of this section will be that $q$ \textit{is not a root of unity}.

Recall (\cite[Proposition 2.3]{cha-pres-qa}) that, as vector spaces, we have a triangular decomposition
\[
\mathcal{A}_{N,q} = U^- \otimes U^0 \otimes U^+\;,
\]
where $U^-,U^0, U^+$ are sub-algebras of $\mathcal{A}_{N,q}$.

Also, recall that the definition of $\mathcal{A}_{N,q}$ involves the invertible elements
\[
k_i\in U^0,\,i=1,\ldots,N-1\;,
\]
which generate a commutative co-commutative Hopf-sub-algebra $K<\mathcal{A}_{N,q}$. This sub-algebra can be seen as the Cartan algebra (zero-part) in the triangular decomposition of the smaller quantum group $U_q(\mathfrak{sl}_N)$, which is naturally embedded in $\mathcal{A}_{N,q}$.

Let $\Omega$ be the set of complex characters $\lambda$ of $K$, which satisfy $\lambda(k_i)\in q^{\mathbb{Z}}$, for all $i=1,\ldots,N-1$. Since $K$ is co-commutative, $\Omega$ is an abelian group. Given a complex character $P$ of the algebra $U^0$, we will write $\omega(P)$ to be its restriction to $K$.

We are further interested in the category $\mathcal{C}^q_N$ of \textit{type $1$} representations of $\mathcal{A}_{N,q}$. It consists of all finite-dimensional representations $V$ in which a certain specified central element $c^{1/2}\in U^0$ acts trivially, and which comply with a weight space decomposition for the sub-algebra $K$, of the form
\[
V= \oplus_{\lambda \in \Omega} V_\lambda\;,
\]
where $V_\lambda$ is the $\lambda$-eigenspace of $V$.

\subsection{Cyclic modules}

For a representation $V$ in $\mathcal{C}^q_N$, we say that a vector $v\in V$ is \textit{highest weight}, if $v$ is an eigenvector for the algebra $U^0$ and $U^+\cdot v = \mathbb{C}\cdot v$. The character of $U^0$ by which it acts on a highest weight vector $v$ is called the $\ell$-weight of $v$ (not to be confused with the previous weight decomposition by characters of the smaller algebra $K$).

Given an irreducible representation $V$ in $\mathcal{C}^q_N$, it is known there is a unique, up to scalar, highest weight vector $v\in V$. The $\ell$-weight $P$ of $v$ characterizes the isomorphism class of $V$. We can write in this case $V = V(P)$.

Let us write $\mathcal{D}_N$ for the set of characters of $U^0$ which give rise to irreducible representations $V(P)$, $P\in \mathcal{D}_N$. The set $\mathcal{D}_N$ can be described using what is known as \textit{Drinfeld polynomials}.

The set $\mathcal{D}_N$ also comes with a natural monoid structure, in the following sense. Given highest weight vectors $v_P\in V,\; v_Q\in W$ of respective $\ell$-weights $P,Q\in \mathcal{D}_N$, the vector $v_P\otimes v_Q\in V\otimes W$ is also a highest weight vector of $\ell$-weight $P\cdot Q\in \mathcal{D}_N$.

Let $V_1,\ldots, V_k$ be irreducible representations in $\mathcal{C}_N$. Let $v_i\in V_i$, $1\leq i\leq k$ be the corresponding highest weight vectors. Let $W\subseteq V_1\otimes \cdots \otimes V_k$ be the subrepresentation generated by the vector $v_1\otimes \cdots\otimes v_k$.

Inspired by the terminology of \cite{hernan-cyc}, we say that the tuple $(V_1,\ldots, V_k)$ is \textit{cyclic} if $W= V_1\otimes \cdots \otimes V_k$.

We would like to reformulate the notion of cyclic tuples into a more categorical language.

\begin{proposition}\label{cyclic}

Let
\[
\underline{V} = (V(P_1),\ldots, V(P_k))
\]
be a tuple of irreducible representations in $\mathcal{C}^q_N$, with $P_1,\ldots,P_k\in \mathcal{D}_N$.

The tuple $\underline{V}$ is cyclic, if and only if, the representation
\[
V(P_1)\otimes \cdots\otimes V(P_k)
\]
has a unique irreducible quotient whose isomorphism class is given by $V(P_1\cdot\ldots\cdot P_k)$.

\end{proposition}

\begin{proof}
For $P\in \mathcal{D}_N$, note that $\omega(P)\in \Omega$ by the assumptions on $\mathcal{C}^q_N$.

Recall (\cite[Theorem 1.3(3)]{fren-mukh}) that $\dim V(P_i)_{w(P_i)}=1$, and that for any $\lambda\in \Omega$ with $V(P_i)_\lambda\neq \{0\}$, we have $w(P_i)> \lambda$. This inequality is meant in a sense of the partial order, which comes from the natural identification of $\Omega$ with the wight lattice of the simple Lie algebra $\mathfrak{sl}_N$.

It easily follows that the representation $ M:=V(P_1)\otimes\cdots \otimes V(P_k)$ has $\dim M_{\omega(P_1)+\ldots +\omega(P_k)}=1$.

Let $v_i\in V(P_i)$, $i=1,\ldots, k$ be the corresponding highest weight vectors, with $\ell$-weights $P_i$. Then, $m:=v_1\otimes\cdots\otimes v_k\in M_{\omega(P_1)+\ldots +\omega(P_k)}$. In particular, $w(P_1\cdot\ldots\cdot P_k) = w(P_1)+\ldots+w(P_k)$.

Now, let $W\subset M$ be the subrepresentation generated by the vector $m$. By standard arguments on weight decompositions, the one-dimensionality of $W_{\omega(P_1 \cdot\ldots\cdot P_k)}=M_{\omega(P_1 \cdot\ldots\cdot P_k)}$ implies that $W$ has a unique irreducible quotient $Z$ to which $m$ has a non-zero projection. In particular, the projection of $m$ in $Z$ is a highest weight vector of $\ell$-weight $P_1\cdot\ldots\cdot P_k$, which means $Z\cong V(P_1\cdot\ldots\cdot P_k)$. One direction of the proposition readily follows.

Conversely, suppose that $M$ has an irreducible quotient $Y\cong V(P_1\cdot\ldots\cdot P_k)$. Again, the one-dimensionality of $M_{\omega(P_1 \cdot\ldots\cdot P_k)}$ requires $m$, and hence $W$, to project non-trivially on $Y$. Since $Y$ is irreducible, we must have $W=M$.

\end{proof}

\begin{theorem}[Hernandez\cite{hernan-cyc}]\label{hern-thm}
Let $V_1,\ldots, V_k$ be irreducible representations in $\mathcal{C}^q_N$, such that for all $1\leq i<j\leq k$, the pair $(V_i,V_j)$ is cyclic. Then, the tuple $(V_1,\ldots, V_k)$ is cyclic.
\end{theorem}

\begin{remark}
The motivation for the above theorem slightly differs from its current application in our work. It was proved as part of a systematic study of the non-semi-simple category $\mathcal{C}^q_N$. Since a cyclic representation whose dual is cyclic as well must be irreducible, the theorem implies an irreducibility criterion for tensor products of representations (which, in fact, was proved earlier in \cite{hern-irr}). Thus, this result generalized the criterion and described a factorization of a family of cyclic representations into a product of irreducible ones.
\end{remark}
\begin{remark}
In our considerations for the problem at hand we will apply Theorem \ref{hern-thm} on a family of representations analogous (through a Schur-Weyl duality, see next section) to quasi-Speh representations as in Proposition \ref{main-prop}. It would be interesting to compare this application to the special case of Theorem \ref{hern-thm} that was presented in \cite{chari-braid}.
\end{remark}

\subsection{Quantum affine Schur-Weyl duality}\label{sect-sw-dual}
Let $q\in \mathbb{C}^\times$ be fixed.
For each $k\geq1,\,N\geq2$, Chari-Pressley \cite{cp-duality} have defined a quantum affine Schur-Weyl duality functor
\[
\mathcal{F}_{k, N}: \mathcal{M}^{q^2}_k \to \mathcal{C}^q_N\;.
\]

Recall from Section \ref{sect-hecke} that $\mathcal{M}^{q^2}_k$ stands for the category of finite-dimensional representations of the affine Hecke algebra $H(k,q^2)$.

These functors served as an affine generalization of the work of Jimbo \cite{jimbo} on a quantum Schur-Weyl duality between quantum groups and (finite) Hecke algebras.

Let us recall some of the properties of these duality functors, as were studied in \cite{cp-duality}.

When $k\leq N$, $\mathcal{F}_{k,N}$ is a fully faithful functor. In other words, it serves as an embedding of the category $\mathcal{M}^{q^2}_k$ into a certain full sub-category of $\mathcal{C}^q_N$. In fact, this sub-category can be easily characterized by an intrinsic property of representations in $\mathcal{C}^q_N$ (these are called \textit{level $k$} representations). We will not need this characterization in our discussion.


Moreover, the duality functors are monoidal (\cite[Proposition 4.7]{cp-duality}), in the sense that
\begin{equation}\label{monoi}
\mathcal{F}_{k_1,N}(\pi_1)\otimes \mathcal{F}_{k_2,N}(\pi_2)\cong \mathcal{F}_{k_1+k_2,N}(\pi_1\times \pi_2),\;
\end{equation}

for all representations $\pi_1,\pi_2$ in $\mathcal{M}_{k_1}^{q^2}, \mathcal{M}_{k_2}^{q^2}$, respectively. Recall, that $\pi_1\times \pi_2$ denotes the induction product which produces a representation in $\mathcal{M}_{k_1+k_2}^{q^2}$.

Let us now return to the setting of the first sections and assume that $q^2 = q_F$. Recall the bijection
\[
\hat{Z}:\multi_0 \to \bigcup_{k\geq0} \irr(\mathcal{M}^{q_F}_k)\;,
\]
that is given by the Zelevinski(-Rogawski) classification.

The following proposition implies, in particular, that for all $k\leq N$, the duality functor $\mathcal{F}_{k,N}$ sends irreducible representations to irreducible representations.

This correspondence of irreducible objects from different categories is rather intriguing. For example, the family of irreducible Speh representations is sent by the composed functor $\mathcal{F}_{k,N}\circ U_\Theta$ ($U_\Theta$ as in Section \ref{sect-equiv-heck}) to the family of \textit{Kirillov–Reshetikhin modules} in $\mathcal{C}^q_N$. The latter is an important family of modules for the study of quantum affine algebras, which was initially explored in \cite{kir-r} (see the survey \cite{chari-survey}, for example).

\begin{proposition}\cite[Theorem 7.6]{cp-duality}\label{irr-corr}

For every $N\geq2$, there is a natural embedding of monoids $\iota_N:\mathcal{D}_N\to \multi_0$, for which
\[
\mathcal{F}_{k_P,N}( \hat{Z}(\iota_N(P)))\cong V(P)\in\irr (\mathcal{C}^q_N)\;
\]
holds, for all $P\in \mathcal{D}_N$ with $k_P\leq N$, where $k_P$ is the integer for which $\hat{Z}(\iota_N(P))\in \irr (\mathcal{M}_{k_P}^{q_F})$.

Moreover, for every $1\leq k\leq N$, the collection of multisegments $\hat{Z}^{-1}(\irr(\mathcal{M}_k^{q_F}))$ is contained in the image of $\iota_N$.

\end{proposition}

We finish by noting that the combination of Proposition \ref{cyclic}, Theorem \ref{hern-thm}, Proposition \ref{irr-corr} and \eqref{monoi} directly give the proof of Theorem \ref{transl}.

\bibliographystyle{alpha}
\bibliography{propo2}{}

\end{document}